\documentclass[12pt]{amsart}
\usepackage{amsmath,amssymb,amsbsy,amsfonts,latexsym,amsopn,amstext,cite,
                                               amsxtra,euscript,amscd,bm}
\usepackage{url}

\usepackage{mathrsfs}

\usepackage[colorlinks,linkcolor=blue,anchorcolor=blue,citecolor=blue,backref=page]{hyperref}
\usepackage{color}
\usepackage{graphics,epsfig}
\usepackage{graphicx}
\usepackage{float}
\usepackage{epstopdf}
\hypersetup{breaklinks=true}

\usepackage[T1]{fontenc}
\usepackage[utf8]{inputenc}

\usepackage[hmargin=3cm,vmargin=3cm]{geometry}

\usepackage{pgfplots}

\pgfplotsset{
	tick label style={font=\footnotesize },
	label style={font=\footnotesize},
}

\usepackage[np]{numprint}
\npdecimalsign{\ensuremath{.}}

\usepackage{bibentry}

\usepackage{tikz}

\usepackage[english]{babel}
\usepackage{mathtools}
\usepackage{todonotes}
\usepackage{url}

\usepackage[norefs,nocites]{refcheck}

\DeclareMathOperator{\vol}{vol}

\begin{document}

\newcommand{\bs}{\boldsymbol}
\def \a{\alpha} \def \b{\beta} \def \d{\delta} \def \e{\varepsilon} \def \g{\gamma} \def \k{\kappa} \def \l{\lambda} \def \s{\sigma} \def \t{\theta} \def \z{\zeta}

\newcommand{\mb}{\mathbb}

\newtheorem{theorem}{Theorem}
\newtheorem{lemma}[theorem]{Lemma}
\newtheorem{claim}[theorem]{Claim}
\newtheorem{cor}[theorem]{Corollary}
\newtheorem{conj}[theorem]{Conjecture}
\newtheorem{prop}[theorem]{Proposition}
\newtheorem{definition}[theorem]{Definition}
\newtheorem{question}[theorem]{Question}
\newtheorem{example}[theorem]{Example}
\newcommand{\hh}{{{\mathrm h}}}

\theoremstyle{definition}
\newtheorem{remark}[theorem]{Remark}

\numberwithin{equation}{section}
\numberwithin{theorem}{section}
\numberwithin{table}{section}
\numberwithin{figure}{section}

\def\sssum{\mathop{\sum\!\sum\!\sum}}
\def\ssum{\mathop{\sum\ldots \sum}}
\def\iint{\mathop{\int\ldots \int}}

\newcommand{\diam}{\operatorname{diam}}

\def\squareforqed{\hbox{\rlap{$\sqcap$}$\sqcup$}}
\def\qed{\ifmmode\squareforqed\else{\unskip\nobreak\hfil
\penalty50\hskip1em \nobreak\hfil\squareforqed
\parfillskip=0pt\finalhyphendemerits=0\endgraf}\fi}

%  use the AMS-Euler Fraktur fonts
%%%%%%%%%%%%%%%%%%%%%%%%%%%%%%%%%%
\newfont{\teneufm}{eufm10}
\newfont{\seveneufm}{eufm7}
\newfont{\fiveeufm}{eufm5}
%%%%%%%%%%%%%%%%%%%%%%%%%%%%%%%%%
%
%  allow automatic size selection in math mode

%
%%%%%%%%%%%%%%%%%%%%%%%%%%%%%%%%%
\newfam\eufmfam
     \textfont\eufmfam=\teneufm
\scriptfont\eufmfam=\seveneufm
     \scriptscriptfont\eufmfam=\fiveeufm
%%%%%%%%%%%%%%%%%%%%%%%%%%%%%%%%%
%
%  \frak works on a single symbol at a time...
%
\def\frak#1{{\fam\eufmfam\relax#1}}

\newcommand{\bflambda}{{\boldsymbol{\lambda}}}
\newcommand{\bfmu}{{\boldsymbol{\mu}}}
\newcommand{\bfxi}{{\boldsymbol{\eta}}}
\newcommand{\bfrho}{{\boldsymbol{\rho}}}

\def\eps{\varepsilon}

\def\fK{\mathfrak K}
\def\fT{\mathfrak{T}}
\def\fL{\mathfrak L}
\def\fR{\mathfrak R}

\def\fA{{\mathfrak A}}
\def\fB{{\mathfrak B}}
\def\fC{{\mathfrak C}}
\def\fM{{\mathfrak M}}
\def\fS{{\mathfrak  S}}
\def\fU{{\mathfrak U}}
\def\fZ{{\mathfrak Z}}
 
\def\sssum{\mathop{\sum\!\sum\!\sum}}
\def\ssum{\mathop{\sum\ldots \sum}}
\def\dsum{\mathop{\quad \sum \qquad \sum}}
\def\iint{\mathop{\int\ldots \int}}
 
\def\T {\mathsf {T}}
\def\Tor{\mathsf{T}_d}
\def\Tore{\widetilde{\mathrm{T}}_{d} }

\def\sM {\mathsf {M}}
\def\sL {\mathsf {L}}
\def\sK {\mathsf {K}}
\def\sP {\mathsf {P}}

\def\ss{\mathsf {s}}

\def\Kmnd{\cK_d(m,n)}
\def\Kmnp{\cK_p(m,n)}
\def\Kmnq{\cK_q(m,n)}

\def \balpha{\bm{\alpha}}
\def \bbeta{\bm{\beta}}
\def \bgamma{\bm{\gamma}}
\def \bdelta{\bm{\delta}}
\def \bzeta{\bm{\zeta}}
\def \blambda{\bm{\lambda}}
\def \bchi{\bm{\chi}}
\def \bphi{\bm{\varphi}}
\def \bpsi{\bm{\psi}}
\def \bxi{\bm{\xi}}
\def \bnu{\bm{\nu}}
\def \bomega{\bm{\omega}}

\def \bell{\bm{\ell}}

\def\eqref#1{(\ref{#1})}

\def\vec#1{\mathbf{#1}}

\newcommand{\abs}[1]{\left| #1 \right|}

\def\Zq{\mathbb{Z}_q}
\def\Zqx{\mathbb{Z}_q^*}
\def\Zd{\mathbb{Z}_d}
\def\Zdx{\mathbb{Z}_d^*}
\def\Zf{\mathbb{Z}_f}
\def\Zfx{\mathbb{Z}_f^*}
\def\Zp{\mathbb{Z}_p}
\def\Zpx{\mathbb{Z}_p^*}
\def\cM{\mathcal M}
\def\cE{\mathcal E}
\def\cH{\mathcal H}

\def\le{\leqslant}

\def\ge{\geqslant}

\def\sfB{\mathsf {B}}
\def\sfC{\mathsf {C}}
\def\sfS{\mathsf {S}}
\def\sfI{\mathsf {I}}
\def\sfJ{\mathsf {J}}
\def\L{\mathsf {L}}
\def\FF{\mathsf {F}}

\def\sE {\mathscr{E}}
\def\sF {\mathscr{F}}
\def\sS {\mathscr{S}}

%%%%%%%%%%%%%%%%%%%%%%%%%
% Alphabet calligraphie %
%%%%%%%%%%%%%%%%%%%%%%%%%
\def\cA{{\mathcal A}}
\def\cB{{\mathcal B}}
\def\cC{{\mathcal C}}
\def\cD{{\mathcal D}}
\def\cE{{\mathcal E}}
\def\cF{{\mathcal F}}
\def\cG{{\mathcal G}}
\def\cH{{\mathcal H}}
\def\cI{{\mathcal I}}
\def\cJ{{\mathcal J}}
\def\cK{{\mathcal K}}
\def\cL{{\mathcal L}}
\def\cM{{\mathcal M}}
\def\cN{{\mathcal N}}
\def\cO{{\mathcal O}}
\def\cP{{\mathcal P}}
\def\cQ{{\mathcal Q}}
\def\cR{{\mathcal R}}
\def\cS{{\mathcal S}}
\def\cT{{\mathcal T}}
\def\cU{{\mathcal U}}
\def\cV{{\mathcal V}}
\def\cW{{\mathcal W}}
\def\cX{{\mathcal X}}
\def\cY{{\mathcal Y}}
\def\cZ{{\mathcal Z}}
\newcommand{\rmod}[1]{\: \mbox{mod} \: #1}

\def\cg{{\mathcal g}}

\def\vy{\mathbf y}
\def\vr{\mathbf r}
\def\vx{\mathbf x}
\def\va{\mathbf a}
\def\vb{\mathbf b}
\def\vc{\mathbf c}
\def\ve{\mathbf e}
\def\vf{\mathbf f}
\def\vg{\mathbf g}
\def\vh{\mathbf h}
\def\vk{\mathbf k}
\def\vm{\mathbf m}
\def\vz{\mathbf z}
\def\vu{\mathbf u}
\def\vv{\mathbf v}

\def\e{{\mathbf{\,e}}}
\def\ep{{\mathbf{\,e}}_p}
\def\eq{{\mathbf{\,e}}_q}

\def\Tr{{\mathrm{Tr}}}
\def\Nm{{\mathrm{Nm}}}

 \def\SS{{\mathbf{S}}}

\def\lcm{{\mathrm{lcm}}}

 \def\0{{\mathbf{0}}}

\def\({\left(}
\def\){\right)}
\def\l|{\left|}
\def\r|{\right|}
\def\fl#1{\left\lfloor#1\right\rfloor}
\def\rf#1{\left\lceil#1\right\rceil}
\def\sumstar#1{\mathop{\sum\vphantom|^{\!\!*}\,}_{#1}}

\def\mand{\qquad \mbox{and} \qquad}

\def\tblue#1{\begin{color}{blue}{{#1}}\end{color}}

%%%%%%%%%%%%%%%%%%%%%%%%%%%%%%%%%%%%%%%%%%%%%%%%%%%%%%%%
%%%%%%%%%%%%%%%%%%%%%%%%%%%%%%%%%%%%%%%%%%%%%%%%%%%%%%%%
%%%%%%%%%%%%%%%%%%%%%%%%%%%%%%%%%%%%%%%%%%%%%%%%%%%%%%%%
%%%%%%%%%%%%%%%%%%%%%%%%%%%%%%%%%%%%%%%%%%%%%%%%%%%%%%%%

%%%%%%%  END OF STANDARD STUFF %%%%%%%%%

%%%%%%%%%%%%%%%%%%%%%%%%%%%%%%%%%%%%%%%%%%%%%%%%%%%%%%%%
%%%%%%%%%%%%%%%%%%%%%%%%%%%%%%%%%%%%%%%%%%%%%%%%%%%%%%%%
%%%%%%%%%%%%%%%%%%%%%%%%%%%%%%%%%%%%%%%%%%%%%%%%%%%%%%%%
%%%%%%%%%%%%%%%%%%%%%%%%%%%%%%%%%%%%%%%%%%%%%%%%%%%%%%%
%%%%%%%%%%%
%%% Spell

\hyphenation{re-pub-lished}

\mathsurround=1pt

\def\bfdefault{b}

\def \F{{\mathbb F}}
\def \K{{\mathbb K}}
\def \N{{\mathbb N}}
\def \Z{{\mathbb Z}}
\def \P{{\mathbb P}}
\def \Q{{\mathbb Q}}
\def \R{{\mathbb R}}
\def \C{{\mathbb C}}
\def\Fp{\F_p}
\def \fp{\Fp^*}

 \def \xbar{\overline x}

\title[Local mean value estimates]{Local mean value estimates for Weyl sums}

\author[J. Brandes]{Julia Brandes}
\address{JB: \ Mathematical Sciences, University of Gothenburg and Chal\-mers Institute of Technology, 412 96 G\"oteborg, Sweden}
\email{brjulia@chalmers.se}

 \author[C. Chen] {Changhao Chen}
\address{CC: \ Center for Pure Mathematics, School of Mathematical Sciences, Anhui University, Hefei 230601, China}
\email{chench@ahu.edu.cn}

 \author[I. E. Shparlinski] {Igor E. Shparlinski}
\address{IES: \ Department of Pure Mathematics, University of New South Wales,
Sydney, NSW 2052, Australia}
\email{igor.shparlinski@unsw.edu.au}

\begin{abstract}   We obtain new estimates -- both upper and lower bounds -- on the mean values of  the Weyl sums over a small box inside of the unit torus. In particular, we refine  recent conjectures of C.~Demeter and B.~Langowski (2022), and improve some of their results. 
\end{abstract}

\subjclass[2020]{Primary: 11L15; Secondary: 11L07, 11D45}

\keywords{Weyl sum, mean value theorem, small box}

\maketitle

\tableofcontents

\section{Introduction} 
 
\subsection{Background and motivation} 
The study of exponential sums occupies a central location in the analytic theory of numbers, as they are a crucial tool connecting the language of number theory with the language of Fourier analysis. In fact, many of the most celebrated results in number theory either are equivalent to or at least crucially depend on strong bounds on exponential sums, either in an average or a pointwise sense. 

In this paper, we are interested in exponential sums of the shape
$$
	S_d(\vx; N) = \sum_{n=1}^N  \e\(x_1n + \ldots  + x_d n^d\), 
$$
associated to Vinogradov's mean value theorem. Thanks to the breakthrough results of Bourgain, Demeter and Guth~\cite{BDG} as well as Wooley ~\cite{Wool1,Wool2}, we now have very good control over the average value of these sums as $\vx$ ranges over the unit hypercube $[0,1)^d$. If we put  
\begin{equation}\label{eq:J(N)} 
	J_{s,d}(N) = 	\int_{[0,1]^d} |S_d(\vx; N)|^{2s}d\vx,
\end{equation} 
then~\cite{BDG, Wool1,Wool2} show that 
\begin{equation} \label{eq:MVT}
J_{s,d}(N)   \le  N^{s+o(1)} +   N^{2s - d(d+1)/2+o(1)},
\end{equation}
which is optimal up to (at most) the $o(1)$. It should be noted that one can show as a consequence 
of~\eqref{eq:MVT} that $S_d(\vx;N) \le N^{1/2 + o(1)}$ for almost all $\vx \in [0,1)^d$, see~\cite[Corollary~2.2]{ChSh-IMRN}. Thus, we have now a close to complete understanding of the size of exponential sums both on average and in an almost-all sense. 

Unfortunately, neither of these results is apt to tell us much about the pointwise size of $S_d(\vx;N)$ for any fixed point $\vx$, and indeed our understanding of this problem is still far from the conjectured bounds. It is not hard to see that such pointwise bounds necessarily depend on the diophantine approximation properties of $\vx$. Suppose that $x_d$ has an approximant $a_d/q$ with $\|q x_d\| \le q^{-1}$, then an argument going back to Vinogradov (see~\cite[Theorem~5.2]{Vau}) shows that the mean value bound~\eqref{eq:MVT} can be used to derive the pointwise estimate 
$$
	|S_d(\vx;N)| \le N^{1+o(1)} \left(N^{-1} + q^{-1} + q N^{-d}\right)^{1/d(d-1)}.
$$
However, in order to make progress towards the bound 
$$
	|S_d(\vx;N)| \le N^{1+o(1)} \left(q^{-1} + q N^{-d}\right)^{1/d}
$$
conjectured in~\cite[Chapter 3, Conjecture 1]{Mont} one likely needs different methods -- although we point out that in the case of one-dimensional exponential sums, a bound of at least comparable quality to the conjectured one, with the exponent $1/d$ replaced by $1/(2d-2)$, would follow from the conjectured mean value (Hua-type) bound for such sums, see~\cite[Theorem~2.1]{BrPa}. 

The purpose of the manuscript at hand is to investigate $S_d(\vx;N)$ and related exponential sums as $\vx$ ranges over small boxes. This should rightly be viewed as an attempt to interpolate between our almost complete understanding of mean values of $S_d(\vx;N)$ and our deficient understanding of the pointwise behaviour of these sums. Our work ties in with work by Demeter and Langowski~\cite{DeLa} as well as some speculations of Wooley~\cite{Wool3}. By introducing several new ideas, based partly on bounds for inhomogeneous Vinogradov systems as explored recently by Brandes and Hughes~\cite{BrHu} as well as Wooley~\cite{Wool3}, and partly on the structure of large Weyl sums investigated in some depth by Baker (see, for example,~\cite{Bak0,Bak1}), we are able to extend and improve some of the results of~\cite{DeLa}. We obtain a diverse zoo of bounds, which we describe and discuss in more detail in Section~\ref{sec:new bounds} below. These bounds have fairly different character depending on the size of the small box. In a sense, this is not unexpected, since the mean value of $S_d(\vx;N)$ over a very small box located at the origin is dominated by the spike at $\vx=\boldsymbol{0}$, whereas the behaviour comes to increasingly resemble that of mean values over the entire unit hypercube as the size of the box increases. What is not clear is how and at what scale(s) the transition between these two behaviours takes place. Taken collectively, our bounds hint that this transition may be more intricate than hitherto anticipated, and we hope that future research can provide a more accurate picture of these phenomena.

\subsection{Set-up} 
For an integer $\nu \ge 1$ we denote by $\T_\nu$ the $\nu$-dimensional unit torus, which we also identify with the $\nu$-dimensional unit cube, that is, 
$$
	\T_\nu = \(\R/\Z\)^\nu = [0,1)^\nu. 
$$
For positive integers $d$ and $N$, a sequence of complex weights $\va = \(a_n\)_{n=1}^N$, and a vector $\vx \in \Tor$, we define the Weyl sums  
$$
	S_d(\vx; \va, N) = \sum_{n=1}^N a_n \e\(x_1n + \ldots  + x_d n^d\) 
$$
where $\e(z) = \exp(2\pi i z)$. 

For a positive $\delta\le 1$ and $\bxi \in \Tor$, we define
\begin{equation}\label{eq:sum-xi}
	I_{s, d}(\delta, \bxi; \va, N)=\int_{\bxi+ [0,\delta]^d} \left| S_d(\vx; \va, N)\right |^{2s}  d\vx.
\end{equation}
We note that the exponent $s$ in~\eqref{eq:sum-xi} is not necessary integer but can take 
arbitrary real positive values. 
The question of estimating $I_{s,d}(\delta, \bxi; \va,N)$ for suitable choices of $\bxi$ and $\va$ has recently received some attention, see, for example,~\cite{CKMS,ChSh-QJM,DeLa,Wool3} for various bounds and applications. 
The case of boxes at the origin is especially interesting. In fact, it is easy to see that the question about the size of $ I_{s,d} (\delta, \bxi; \va, N) $ can be reduced to  $I_{s,d} (\delta, \boldsymbol{0};  \widetilde \va, N)$, with $\widetilde a_n = a_n  \e\(\xi_1n + \ldots + \xi_d n^d\)$ for $n =1, \ldots, N$.  We thus put 
$$
	I_{s, d}^{(0)}(\delta; \va, N)= I_{s, d}(\delta, \boldsymbol{0}; \va, N).
$$
Hence in the case of arbitrary weights, without loss of generality, it suffices to study  
the quantity $I_{s,d}^{(0)} (\delta; \va, N)$.

Meanwhile, arguably the most relevant choice of weights $\va$ is that in which $a_n = 1$ for $n \le N$, so we consider this situation separately. Thus, in the case when $\va = \boldsymbol 1$, we define
$$
	I_{s,d}^{(0)} (\delta; N) = I_{s, d}^{(0)}(\delta; \boldsymbol{1}, N),
$$
as well as
$$
	I_{s, d}^{\sharp}(\delta; N)=\sup_{\bxi\in \T_d} I_{s, d}(\delta, \bxi; \boldsymbol 1, N) \mand I_{s, d}^{\flat}(\delta; N)= \inf _{\bxi \in \T_d} I_{s, d}(\delta, \boldsymbol{\xi}; \boldsymbol 1, N).
$$
Note that since the unit torus $\T_d = (\R/\Z)^d$ is compact as an additive group, the infimum and supremum here are actually attained as the exponential sum is continuous. 

By the discussion following~\eqref{eq:sum-xi} it is easy to see that 
\begin{equation}\label{eq:I and I}
I_{s, d}^{\sharp}(\delta; N)  \le  \sup_{\|\va\|_\infty\le 1} I_{s, d}^{(0)}(\delta; \va, N)
\end{equation}
where  supremum is taken over all sequences of complex weights
with $\|a_n\|_{\infty} \le 1$.

\subsection{Notation} 

Throughout the paper, we use the Landau and Vinogradov notations $U = O(V)$, $U \ll V$ and $ V\gg U$ to express that $|U|\leqslant c V$ for some positive constant $c$, which throughout the paper may depend on the degree $d$ and occasionally on the small real positive parameter $\varepsilon$ and the arbitrary real parameter $t$. We also write $U \asymp V$ as an equivalent of $U \ll V \ll U$. 
Moreover, for any quantity $V> 1$ we write $U = V^{o(1)}$ (as $V \rightarrow \infty$) to indicate a function of $V$ which satisfies $ V^{-\eps} \le |U| \le V^\eps$ for any $\eps> 0$, provided $V$ is large enough. One additional advantage of using $V^{o(1)}$ is that it absorbs $\log V$ and other similar quantities without changing the whole expression.  

We also recall the definition of the $\ell^p$-norm, which for a sequence of complex numbers $\va = (a_n)_{1 \le n \le N}$ and a real number $p \ge 1$ is given by 
$$
\|\va\|_p = \( \sum_{n=1}^N |a_n |^p\)^{1/p}.
$$
 
For $m \in \N$, we write $[m]$ to denote the set $\{0, 1, \ldots, m-1\}$.
We denote the cardinality of a finite set $\cS$ by $\#\cS$, and for a measurable set $\cT \subseteq \T_\nu$ we write $\lambda(\cT)$ for the Lebesgue measure of the appropriate dimension $\nu$. 

We  use the notation $\lfloor x \rfloor$ and $\lceil x \rceil$ for the largest integer no larger than $x$ and the smallest integer no smaller than $x$, respectively. We then write $\{x\}=x- \lfloor x \rfloor \in [0,1)$.

\section{What  we know and what we believe to be true}

\subsection{State of the art and previous conjectures}
In order to get a better sense of what to expect, it is helpful to first record some known bounds that can serve as a benchmark for our ensuing considerations. 
On the one hand, when $\delta=1$ the recent advances of  Bourgain, Demeter and Guth~\cite{BDG} and Wooley~\cite{Wool2} towards the  optimal form of the Vinogradov mean value theorem yield the bound 
\begin{equation} \label{eq:MVT-a}
	I_{s, d}^{(0)}(1; \va, N) \le    \|\va\|_2^{2s}N^{o(1)}(1+N^{s-s(d)})
\end{equation} 
for all $s>0$, where   
$$
	s(d)=d(d+1)/2. 
$$ 
For general $\va$ this is essentially sharp, since for $\va = \boldsymbol{1}$ a standard argument shows that 
\begin{equation}\label{eq:MVT-lower} 
	I^{(0)}_{s, d}(1;N)  = J_{s,d}(N)  \gg N^{s}+ N^{2s-s(d)},
\end{equation}	 
where $J_{s,d}(N)$ is given by~\eqref{eq:J(N)}. In fact, by adapting the argument of~\cite[Lemma~3.1]{CKMS} one can show that $S_d(\vx;N) \gg N^{1/2}$ for a positive proportion of $\vx \in \T_d$.

On the other hand, for very small values of $\delta$ we can bound the integral trivially and obtain  
\begin{equation}\label{eq: Triv Bound zero} 
	I_{s, d}^{(0)}(\delta;\va, N) \le \delta^d \| \va \|_1^{2s}. 
\end{equation} 
By a slightly more sophisticated argument, combining the bound of~\eqref{eq:MVT-a} with H{\"o}lder's inequality, we obtain the bound
\begin{equation} \label{eq: Triv Bound}
	I_{s, d}^{(0)}(\delta;\va, N)  \le \delta^{d-2s/(d+1)}\|\va\|_2^{2s}   N^{o(1)}, \qquad  0 \le s  \le s(d),
\end{equation}
see also~\cite[Equation~(2.3)]{DeLa}. Clearly, in the limit $\delta \to 1$, as expected, the bound~\eqref{eq: Triv Bound} approaches the  bound~\eqref{eq:MVT-a}. At the same time, we see that for small $\delta$ this is weaker than the trivial bound~\eqref{eq: Triv Bound zero}. 

In the special case $s=2$, a further example can be derived from~\cite[Lemma~4.5]{CKMS}, which implies that if   $|a_n| \le 1$ for $n =1, \ldots, N$,  then
$$ 
	I^{(0)}_{2, d}  (\delta ; \va, N)  \le \delta^d N^{2} + \delta^{d-4} N^{1+o(1)}.
$$

For lower bounds, observe that for any $N$ we have 
$$
	I^{(0)}_{s, d}(1;N) \ll \delta^{-d} I_{s, d}^{\sharp}(\delta;N).
$$ 
Upon combining this with the classical lower bound of~\eqref{eq:MVT-lower}, we thus conclude that  
$$
I_{s, d}^{\sharp}(\delta;N)\gg  \delta^d (N^s + N^{2s-s(d)}). 
$$
Clearly, this suggests the question of whether this bound is sharp, and if so, in what ranges. A version of that conjecture has been proposed in recent work by 
Wooley~\cite[Conjectures~8.1 and~8.2]{Wool3}.

\begin{conj}[Wooley~\cite{Wool3}]\label{conj:Wooley}
	Suppose that 
	$$
		s \ge \tfrac14d(d+1)+1 \qquad  \text{or} \qquad \delta \ge N^{1/d - (d+1)/4}.
	$$
	 Then 
	$$
	 	I_{s, d}^{\sharp}(\delta;N)\le  \delta^d N^{s+o(1)} + N^{2s-s(d) + o(1)}.
	$$ 
\end{conj}

In Wooley's setting~\cite{Wool3}, the bound on the number of variables is motivated by considerations concerning the convergence of the singular series; however, it seems not unreasonable that the validity of the bound in Conjecture~\ref{conj:Wooley} in the $\delta$-aspect might extend below the proposed range. We also remark that Wooley allows for general measurable sets, whereas we restrict to axis-aligned hypercubes.

Another conjecture that is relevant to our work, and which permits arbitrary positive values of $\delta$ and $s$, has been fielded in recent work by Demeter and Langowski~\cite[Conjecture~1.3]{DeLa}. 
\begin{conj}[Demeter--Langowski~\cite{DeLa}]\label{conj:D-L}
	Let 	
	$$
	\rho(d) =  \rf{3d^2/4} -1.
	$$
	We have
	\begin{equation} \label{eq: Conj Bound}
	I_{s, d}^{(0)}(\delta ; \va, N)  \le   \delta^{(d+1)/2} \|\va\|_2^{2s}\(1+N^{s-\rho(d)/2}\) N^{o(1)}. 
	\end{equation}
\end{conj}

By~\cite[Theorem~2.4]{DeLa}  we have~\eqref{eq: Conj Bound} for $d=2$ and $d=3$ in the full range. Moreover, the authors establish bounds of a similar quality also for $d=4$ and $d=5$. 
We also remark that there is nothing intrinsically special about the power of $\delta$ occurring in~\eqref{conj:D-L} or the concomitant value $\rho(d)$. Rather, it seems that the precise formulation and choice of parameters of Conjecture~\ref{conj:D-L} were chosen mostly in view of applications to the mean value of Weyl sums along curves, see~\cite[Proposition~2.2]{DeLa}.

A comparison of Conjectures~\ref{conj:Wooley} and~\ref{conj:D-L} shows that neither is strictly stronger than the other; rather they make different predictions for various ranges of $s$ and various values of $\delta$. It is apparent from the discussion preceding Conjecture~\ref{conj:Wooley} that it is sharp for small $s$ and $\delta$ not too small. At the same time, we remark that Conjecture~\ref{conj:D-L}, if correct, is the best possible in the sense that the exponent $(d+1)/2$ cannot be increased  if one wants a bound which holds for all $\delta \in (0,1)$. Evidence for this has been given in~\cite{DeLa}, after the formulation of~\cite[Conjecture~1.3]{DeLa}. Moreover, for \emph{extremely} small values of $\delta$, the trivial bound~\eqref{eq: Triv Bound zero} is both sharp and stronger than~\eqref{eq: Conj Bound}.
It is therefore an interesting question to derive even a valid heuristic for the behaviour of $I^{(0)}_{s,d}(\delta;\va, N)$ that reflects the true expected size of the quantity for all choices of $\delta$ and $N$.

\subsection{An upper bound for a small cube at the origin and some new conjectures} 
Before embarking on a precise discussion of our results, we remark on a general fact concerning the behaviour of mean values of the type considered in this paper. Typically, for fixed  parameters $d$ and $\delta$, we endeavour to establish bounds of the shape 
$$
	I_{s,d}^{(0)}(\delta; \va,N) \le \delta^{d-\alpha} \|\va\|_2^{2s}(1+N^{s-\sigma_0})  N^{o(1)}
$$  
for some $\alpha \in [0,d]$ and some $\sigma_0\ge 1$ depending on $d$ and $\delta$. In particular, if we can establish such a bound at the critical point $s=\sigma_0$, the corresponding results for the {\it subcritical\/}  and {\it supercritical\/} ranges $s<\sigma_0$ and $s>\sigma_0$ follow by standard arguments. In this paper we  give bounds applicable to both the sub- and supercritical ranges. 

Our first result provides a lower bound for the mean value of Weyl sums over a small cube at the origin.  The proof, which is based on the continuity of Weyl sums $S_{d}(\vx; N) $ as functions of $\vx$, is rather straightforward. We then use this simple bound as a benchmark and a basis for several conjectured upper bounds. It also motivates our results in Section~\ref{sec:new bounds}, which are based on a variety of new ideas. 

We define
\begin{equation}\label{eq: sigma_d}
	\sigma_d(\alpha)= \frac{\alpha(2d - \alpha + 1)-\{\alpha\}(1-\{\alpha\})}{2}.
\end{equation}

\begin{theorem}\label{thm:conj}
	Let 
	$$
		s_0(d, \alpha) = \sup\left\{s \ge 0:~ I^{(0)}_{s,d} (\delta;  N) \le \delta^{d-\alpha} N^{s+o(1)}, \ \forall
		\delta \in [N^{-d}, 1], \text{ as } N \to \infty\right\} .
	$$
	We then have 
	$$
		s_0(d, \alpha) \le \sigma_d(\alpha). 
	$$
\end{theorem}

By our above discussion, the conclusion of Theorem~\ref{thm:conj} can be used to derive bounds on $I^{(0)}_{s,d} (\delta;  N)$ for general values of $s$. In fact, for $s>\sigma_d(\alpha)$ we obtain 
\begin{equation}\label{eq:alpha-bd}
	I^{(0)}_{s,d} (\delta;  N) \le \delta^{d-\alpha}  N^{2s -\sigma_d(\alpha)+o(1)}.
\end{equation}
Meanwhile, for $0<s<\sigma_d(\alpha)$, our Theorem~\ref{thm:conj} in combination with H\"older's inequality yields
$$
	I^{(0)}_{s,d} (\delta;  N) \le \delta^{d-\alpha s/ \sigma_d(\alpha)}N^{s+o(1)}.
$$

To put this into context, we compare Theorem~\ref{thm:conj} with our preceding discussion. Consider first the case $\alpha=0$, for which $\sigma_d(0)=0$. Consequently, for any $s$ the bound~\eqref{eq:alpha-bd} reduces to~\eqref{eq: Triv Bound zero}. Meanwhile, taking $\alpha=d$ we obtain $\sigma_d(d)=d(d+1)/2$, which we also know to be sharp when $\delta=1$. Finally, the value $\alpha=(d-1)/2$ produces the bound 
$$
	\sigma_d((d-1)/2) = \frac{3(d^2-1)}{8} -  \frac{\{(d-1)/2\}(1-\{(d-1)/2\})}{2}  = \frac12\left(\rf{3d^2/4} -1\right),
$$
which recovers Conjecture~\ref{conj:D-L} by Demeter and Langowski~\cite[Conjecture~1.3]{DeLa}. In this way, Theorem~\ref{thm:conj} suggests a natural extension of Conjecture~\ref{conj:D-L}.   

\begin{conj}\label{conj1}
	Fix $\alpha \in [0,d]$. 
	 For any  sufficiently large $N$ and any $\delta$ in the range $N^{-d} \le \delta \le 1$, the bound 
	$$
		I^{(0)}_{s,d} (\delta;  \va, N) \le \delta^{d-\alpha} \| \va \|_2^{2s} (1+N^{s-\sigma_d(\alpha)})N^{o(1)}
	$$
	holds for all $s \ge 0$.
\end{conj}

We note that we do not suggest that Conjecture~\ref{conj1} is always sharp, and there are situations where we do, in fact, obtain stronger upper bounds, as can be gleaned from 
Figures~\ref{fig:I22}, \ref{fig:I33} and~\ref{fig:I32} below. 
For $\delta < N^{-d}$ it is not hard to see that the trivial bound~\eqref{eq: Triv Bound zero} gives a stronger result. We also note that a careful inspection of the proof of Theorem~\ref{thm:mvt-weight-conj} shows that for any given $\alpha>0$ Conjecture~\ref{conj1} is sharp at the point $\delta = N^{-\lfloor d-\alpha \rfloor-1}$.

The presence of the additional parameter $\alpha$ in these considerations is somewhat irritating. One checks easily that 
\begin{equation}\label{eq:alpha-bound}
	\sigma_d(\alpha)=\alpha d \qquad \text{for all $\alpha \in (0,1]$.}
\end{equation}
For general values of $\alpha$, one can show by a modicum of computation that $\sigma_d(\alpha)$ is continuous and strictly increasing in $\alpha$ for $\alpha \in [0,d]$. Indeed, we clearly have 
$$\(\frac12\alpha(2d - \alpha + 1)\)' = d-\alpha+1/2,
$$ 
while $ \frac12\{\alpha\}(1-\{\alpha\})$ is the periodic continuation of the function $u(1-u)/2$ for $u \in [0,1)$, and this latter function has derivative $-u+1/2 \in [-1/2, 1/2)$, so that the whole function $\sigma_d(\alpha)$ is continuous and satisfies $\sigma_d'(\alpha) >0$ for all non-integer $\alpha < d$.

For a fixed value $s$, denote by $\alpha_0(d,s)$ the unique $\alpha$ for which $\sigma_d(\alpha)=s$. In this notation, we can change perspective and propose a reformulation of the above conjecture in which we seek to determine the optimal value of $\alpha$ for any given set of parameters $s$ and $d$. 

\begin{conj}\label{conj2}
	For any parameters $d$ and $s\le s(d)$, and for any sufficiently large $N$ and any $\delta$ in the range $N^{-d} \le \delta \le 1$,  we have  
	$$
		I^{(0)}_{s,d} (\delta;  \va, N) \le \delta^{d-\alpha_0(d,s)} \| \va \|_2^{2s} N^{o(1)}.
	$$
\end{conj}

Unfortunately, the function $\alpha_0(d,s)$ is not straightforward to describe explicitly. However, we can give a rough indication of its size. Recalling~\eqref{eq: sigma_d}, write 
\begin{equation}\label{eq:alpha(d,s)}
	\sigma_d(\alpha) = \alpha(2d-\alpha+1)/2- \omega,
\end{equation} 
and note that $\omega = \{\alpha\}(1-\{\alpha\})/2 \in [0,1/8]$. Upon solving~\eqref{eq:alpha(d,s)} for $\alpha$ and substituting $\sigma_d(\alpha)=s$ we obtain that 
$$
	\alpha_0(d,s) = d + 1/2 - \sqrt{ d(d+1)-2s + \nu },
$$
where $\nu=1/4-2\omega \in [0,1/4]$. 
With these considerations, for $s < s(d)$, the bound in Conjecture~\ref{conj2} can be seen to be of the size  
$$
	I^{(0)}_{s,d} (\delta;  \va, N) \le \delta^{\sqrt{ 2s(d)-2s} -1/2 + \eta(d,s)} \| \va \|_2^{2s} N^{o(1)},
$$ 
where 
$$
\eta(d,s)\le\frac{c}{\sqrt{\(s(d) -s\)}}
$$ 
for some absolute constant $c> 0$.

Finally, we remark that Theorem~\ref{thm:conj} as well  both Conjectures~\ref{conj1} and~\ref{conj2}  address only the range $\delta \ge N^{-d}$. However, for smaller $\delta$ it is not hard to show that the bound~\eqref{eq: Triv Bound zero} is sharp. We  give some details on this fact after the proof of Theorem~\ref{thm:conj} below.

\section{New bounds} 
\label{sec:new bounds} 

\subsection{Bounds on mean values with weights}

We first present a family of bounds that can be obtained by combining~\cite[Lemma~3.8]{CKMS} with a result of Wooley~\cite[Theorem~1.3]{Wool3}, which improves a previous result of Brandes and Hughes~\cite{BrHu}.

\begin{theorem}\label{thm: mv weights small s}
	Suppose that $\| \va \|_{\infty} \le 1$ and $0<s\le s(d)/2$. Suppose that $N^{-1}\ge\delta > N^{-d}$, and let $k$ be the unique integer satisfying $N^{-k-1}  < \delta \le N^{-k}$. We then have 
	$$
		I_{s, d}^{(0)}(\delta;\va, N)  \le \delta^{(d+k)/2}N^{s + s(k)/2 + o(1)}.
	$$ 
	Meanwhile, for $\delta>N^{-1}$ we have the bounds
	$$
		I_{s, d}^{(0)}(\delta;\va, N)  \le \begin{cases}\delta^{d/2}N^{s+o(1)} & \text{ for }N^{-1}<\delta<N^{-1/d},\\
		N^{s-1/2+o(1)} & \text{ for }N^{-1/d}<\delta<N^{-1/(2d-1)}, \\
		\delta^{d-1/2}N^{s+o(1)} & \text{ for }N^{-1/(2d-1)}<\delta<1.
		\end{cases}
	$$ 	
\end{theorem} 
We remark that for $\delta \le N^{-d}$ the same methods yield the bound 
$$
	I_{s, d}^{(0)}(\delta;\va, N)  \le \delta^{d}N^{s + s(d)/2 + o(1)},
$$
which is weaker than the trivial bound~\eqref{eq: Triv Bound zero} by our assumption that $s \le s(d)/2$. Since~\eqref{eq: Triv Bound zero} is sharp for small $\delta$, it is worth mentioning that the two bounds coincide at the point $s=s(d)/2$. The interested reader may also note that the range of validity of Theorem~\ref{thm: mv weights small s} covers values of $s$ and $\delta$ for which Conjecture~\ref{conj:Wooley} does not apply. 

For larger values of $s$ we have the following more complicated bound. 
\begin{theorem}\label{thm: mv weights large s}
	For any integer $s$ in the range $s(d)/2 < s < s(d)$ and for any $\delta \ge N^{-1}$, we have 
	$$
		I_{s, d}^{(0)}(\delta;\va, N)  \le N^{s+o(1)}  \(\delta^{d-1} +\sum_{j=1}^{d-1} \min\{\delta^{j-1}(N^{-1/2}+N^{-\eta_{s,d}(j)}), \delta^{(d+j-1)/2}N^{s-s(d)/2} \}\),
	$$
		where 
	\begin{equation}\label{eq:def-eta}
		\eta_{s,d}(\ell)= \(s(d)- s\)\frac{d-\ell+1}{d+\ell +1} \qquad (1 \le \ell \le d-1).
	\end{equation}  
\end{theorem}

Unfortunately, the fairly general bound of Theorem~\ref{thm: mv weights large s} may be somewhat  
hard to parse. However, we note that by always taking the second term in the minimum we obtain the following simple bound.
\begin{cor}
\label{cor:rough bound}
	For any integer $s$ in the range $s(d)/2 < s < s(d)$ and for any $\delta \ge N^{-1}$ we have 
	$$
		I_{s, d}^{(0)}(\delta;\va, N)  \le \delta^{d/2}N^{2s-s(d)/2 + o(1)}.
	$$
\end{cor} 

Similarly, by using always the first expression in the minimum, one can show with a modicum of calculations that in the range $s(d)/2 < s < s(d)$ and for all  $\delta  \le N^{-1/(2d-2)}$ one has
$$
	I_{s, d}^{(0)}(\delta;\va, N)  \le N^{s-1/2+o(1)}. 
$$

Clearly, the bound of Corollary~\ref{cor:rough bound} is not very strong in terms of $\delta$, so for the convenience 
of the reader we state a further corollary to Theorem~\ref{thm: mv weights large s} 
concerning the range of $\delta$ in which the first term dominates. While by no means being deep, this consequence of our result needs some more notation to state.

For a function 
\begin{equation}\label{eq: func f}
	f(x) =  \frac{d+1-x}{(d+x+1)(d-x)}
\end{equation}
define the parameter $\vartheta(d)$ by putting 
\begin{equation}\label{eq: theta}
\begin{split}
	\vartheta(d) &  =  \min\biggl\{ f\(d+1 - \fl{\sqrt{2(d+1)}}\),  f\(d+1 - \rf{\sqrt{2(d+1)}}\) \biggr\}. 
\end{split} 
\end{equation}
In particular, we see that 
$$
	\vartheta(d) \sim \frac{1}{2d} \qquad (d \to \infty).
$$
A list of explicit values of $\vartheta(d)$ for $2 \le d \le 10$ is given in Table~\ref{tab: theta}. 

\begin{table}[H]
	\centering
	\begin{tabular}{|c||c|c|c|c|c|c|c|c|c|}
		\hline
		$d$ & $2$  & $3$ & $4$ & $5$ & $6$ & $7$ & $8$ & $9$ & 10 \\ \hline
		$\vartheta(d) $ & $1/2$  & $3/10$ & $3/14$ & $1/6$ & $2/15$ & $1/9$ & $2/21$ & $1/12$ & $5/68$\\
		\hline
	\end{tabular} 
	\caption{Values of $\vartheta(d)$ for $d =2, \ldots, 10$}
	\label{tab: theta}
\end{table}

\begin{cor}\label{cor: mvt weights large s and delta} 
	Let $d \ge 2$ and recall the definition of $\vartheta(d)$ from~\eqref{eq: theta}. Furthermore, fix some integer $s(d)/2<s < s(d) $ and a sequence of weights satisfying $\|\va\|_\infty\le 1$. 
	Suppose that 
	$$
		\delta > \max\{N^{-1/(2d-2)}, N^{-(s(d)-s) \vartheta(d)}\},
	$$ 
	then
	$$
		I_{s, d}^{(0)}(\delta;\va, N)  \le \delta^{d-1}N^{s+o(1)}.
	$$ 
\end{cor}
	
The proofs of Theorems~\ref{thm: mv weights small s} and~\ref{thm: mv weights large s} depend crucially on the existence of non-trivial bounds for certain inhomogeneous Vinogradov systems. For $\vh = (h_1, \ldots, h_d)\in \Z^d$ let $J_{s,d} (\vh; N)$ be the number of solutions 
to the system of $d$ equations
\begin{equation}\label{eq:equations}
	\sum_{j=1}^{2s} (-1)^j n_j^i = h_i \qquad (i =1, \ldots, d),
\end{equation}
in integer variables $1\le  n_1,\ldots,n_{2s} \le N$. By the triangle inequality, we trivially have  
\begin{equation} \label{eq:Triv Jh}
	J_{s,d} (\vh; N) \le  
 	J_{s,d} (N) \le N^{s + o(1)},
\end{equation}  
where in the last step we have used the classical Vinogradov mean value bound of~\cite[Theorem~1.1]{BDG} in the subcritical range $s \le s(d)$, see~\eqref{eq:MVT}. For most choices of $\vh$, recent results by Brandes and Hughes~\cite{BrHu} and Wooley~\cite{Wool3} give some slight improvement over this in the entire subcritical range. However, the bounds of their work are not expected to be sharp, and indeed one may be tempted to conjecture that for  all integers $s$ in some range $s \le s_1(d)$, for $s_1(d)\le s(d)-1$, one has the stronger bound 
\begin{equation} \label{eq:Conj Jh}
	\max_{\vh \neq \bm 0} J_{s,d} (\vh; N) \le N^{s-\nu+o(1)}   
	\end{equation} 
for some $\nu \in (0,1]$.  
Clearly, the sharpest version of the conjecture in~\eqref{eq:Conj Jh} is the one corresponding to the parameters $\nu=1$ and $s_1(d)=s(d)-1$. Note that for $\nu>1$ the bound~\eqref{eq:Conj Jh} is false even for small values of $s$, as can be seen by choosing $n_1, n_2$ and $\vh$ such that $n_1^j-n_2^j=h_j$ for $1 \le j \le d$, thus reducing the system~\eqref{eq:equations} to a homogeneous system in $2(s-1)$ variables which has $J_{s-1,d} (N) \gg N^{s-1}$ solutions. 
However, the set of possible choices for $\vh$ for which the bound~\eqref{eq:Conj Jh} is sharp with $\nu=1$ is fairly small. Consequently, in many cases we obtain stronger results by averaging over the $\vh$ (see Lemma~\ref{lem:Inhom VMVT-Cauchy} below).

Conditionally on~\eqref{eq:Conj Jh} being known for $\nu=1$, we have the following.

\begin{theorem}\label{thm:mvt-weight-conj}
	Let $d \ge 2$  and $\|\va\|_\infty\le 1$. Assume that~\eqref{eq:Conj Jh} holds with $\nu=1$ for all $s$ in some range $s \le s_1(d)$.  Let  $1\ge\delta > N^{-d}$, and let $k$ be the unique integer satisfying $N^{-k-1}  < \delta \le N^{-k}$.    
	\begin{enumerate}
		\item \label{thm-item:mvt-weight-conj small s}
		Suppose that $0<s\le \min\{s(d)/2, s_1(d)\}$. 
		\begin{itemize} 
			\item For $k \ge 1$, we have
			$$
				I_{s, d}^{(0)}(\delta;\va, N)  \le \delta^{(d+k)/2}N^{s + s(k)/2 + o(1)}.
			$$ 
			\item For $k  = 0$, we have
			$$
				I_{s, d}^{(0)}(\delta;\va, N)  \le \begin{cases}\delta^{d/2}N^{s+o(1)} & N^{-1}<\delta\le N^{-2/d},\\
				N^{s-1+o(1)} & N^{-2/d}<\delta\le N^{-1/d},\\
				\delta^dN^{s+o(1)} & N^{-1/d}<\delta\le 1.
				\end{cases}
			$$ 	
		\end{itemize}
		\item \label{thm-item:mvt-weight-conj large s}
		Suppose now that $s(d)/2 < s \le s_1(d)$.
		For $k \ge 0$, we have
		$$
			I_{s, d}^{(0)}(\delta;\va, N)  \le N^{s+o(1)}\left(\delta^d  +  \min\{\delta^{k}N^{s(k)-1}, \delta^{(k+d)/2}N^{s-(s(d)-s(k))/2}\}\right).
		$$ 
	\end{enumerate}
\end{theorem}
We remark that Wooley's range for $s$ coincides with that in part~\eqref{thm-item:mvt-weight-conj large s} of Theorem~\ref{thm:mvt-weight-conj} when $d \equiv 0$ or $d \equiv 3 \pmod 4$, while for $d\equiv 1$ and $d \equiv 2 \pmod 4$ the value $s=(s(d)+1)/2$ is not covered by~\cite[Conjecture~8.1]{Wool3}, whereas our result is applicable. This is in fact the situation in the (otherwise well-understood) case $d=s=2$, which we discuss below as an example.   

Unfortunately, proving~\eqref{eq:Conj Jh} seems to be quite delicate in general even for non-optimal values of $\nu$. In some special cases, however, suitable bounds are available. For instance, Dendrinos, Hughes and Vitturi~\cite[Lemmas~5 and~6]{DHV} showed that~\eqref{eq:Conj Jh} holds  with $\nu=1$ in the cases $d=s=2$ (which implies the statement for $(s,d)=(3,2)$) and $d=s=3$. Thus, after a comparison of all terms in Theorem~\ref{thm:mvt-weight-conj}, in combination also with~\eqref{eq: Triv Bound zero}, we obtain the following unconditional bounds.  

\begin{cor}\label{cor:mvt-weight-small d} 	Let $\|\va\|_\infty\le 1$. 
	For $s=d=2$ as well as $d=3$ and $s=2$ or $s=3$ the mean value $I_{s,d}(\delta; \va,N)$ is bounded above as detailed in Table~\ref{tab:mvt-weight-small d}.
	\begin{table}[H] 
	\begin{flushleft}
	\begin{tabular}{|l||c|c|c|c|}
		\hline$\delta  $ &$(0, N^{-2}]$ & $(N^{-2}, N^{-1}] $&$ (N^{-1}, N^{-1/2}] $&$ ( N^{-1/2},1]$ \\ \hline
		$I_{2,2}^{(0)}$ & $\delta^2N^{4+o(1)}$ & $\delta N^{2+o(1)}$ & $N^{1+o(1)} $ & $\delta^2 N^{2+o(1)}$ \\ \hline
	\end{tabular} \medskip
		
	\begin{tabular}{|l||c|c|c|c|c|}
		\hline	$\delta $ &$(0, N^{-\frac32}]$ & $(N^{-\frac32}, N^{-1}] $&$ (N^{-1}, N^{-\frac23}]$ & $(N^{-\frac23},N^{-\frac13}]$ & $ ( N^{-\frac13},1]$ \\ \hline
		$I_{2,3}^{(0)}$ & $\delta^3N^{4+o(1)}$ & $\delta^2 N^{\frac52+o(1)}$ &$\delta^{\frac32}N^{2+o(1)}$ & $N^{1+o(1)} $ & $\delta^3 N^{2+o(1)}$ \\ \hline
	\end{tabular} \medskip
	
	\begin{tabular}{|l||c|c|c|c|c|c|}
		\hline $\delta $ &$(0, N^{-3}]$ & $(N^{-3},N^{-2}]$ & $(N^{-2}, N^{-1}] $&$ (N^{-1}, N^{-\frac23}]$ & $(N^{-\frac23},N^{-\frac13}]$ & $ ( N^{-\frac13},1]$ \\ \hline
		$I_{3,3}^{(0)}$ & $\delta^3N^{6+o(1)}$ & $\delta^{\frac52}N^{\frac92+o(1)}$ & $\delta^2 N^{\frac72+o(1)}$ &$\delta^{\frac32}N^{3+o(1)}$ & $N^{2+o(1)} $ & $\delta^3 N^{3+o(1)}$\\ \hline
	\end{tabular}
	\end{flushleft}
	\caption{Upper bounds for $\sup_{\|\va\|_\infty \le 1}I_{s,d}^{(0)}(\delta;\va,N)$ for selected choices of $s$ and $d$, with $\delta$ in corresponding intervals.}
	\label{tab:mvt-weight-small d}
	\end{table}
\end{cor}

For comparison, in the special case $\va = \bm 1$, the conjecture proposed by Wooley~\cite{Wool3}  (Conjecture~\ref{conj:Wooley}) claims that  
\begin{align*}
	I^{\sharp}_{2,2}(\delta; N) &\le \delta^2 N^{2+o(1)} \qquad   \text{for } \delta \ge N^{-1/4}, \\
	I^{\sharp}_{2, 3}(\delta; N) &\le \delta^3 N^{2+o(1)} \qquad   \text{for } \delta \ge N^{-2/3}, \\
	I^{\sharp}_{3, 3}(\delta; N) &\le \delta^3 N^{3+o(1)} \qquad   \text{for } \delta \ge N^{-2/3}.
\end{align*}  
Clearly, the range of applicability here is much smaller than that of our setting, and for $d=2$ Corollary~\ref{cor:mvt-weight-small d} establishes the bound conjectured by Wooley in a much larger range than suggested in~\cite{Wool3}. For $d=3$, we establish the bounds from Conjecture~\ref{conj:Wooley} in the range $N^{-1/3}\le \delta \le 1$, but fall short in the range $N^{-2/3} \le \delta < N^{-1/3}$.

\subsection{Bounds on mean values with shifts}
When $\delta$ is  not too small, we also have some results that stem from
exploiting the structure of large Weyl sums. 

\begin{theorem}
\label{thm:mvt d=2}
For any $s>0$ and any $\delta\ge N^{-3/(6+2s)}$, we have 
$$
I^{\sharp} _{s,2} (\delta;  N) \le \delta^2 N^{2s\(1- 3/(6+2s)\)+o(1)}.
$$
\end{theorem}

For $d\ge 3$ we put 
$$
D = \min\{2^{d-1}, 2d(d-1)\}.
$$ 
We then have the following.

\begin{theorem}\label{thm:mvt d>2} 
	For any $s>(s(d)D-d^2-1)/2$ and $\delta\ge N^{-(d+1)/(2(2s+d^2+1))}$, we have 
	$$
 		I^{\sharp}_{s, d}(\delta;N) \le \delta^d N^{2s\(1-s(d)/(2s+d^2+1)\)+o(1)}.
	$$
\end{theorem}

For context, note that when $\delta$ assumes the smallest possible value, the upper bounds in Theorems~\ref{thm:mvt d=2} and~\ref{thm:mvt d>2} take the shape 
$$
	I_{s, 2}^{\sharp}(\delta;N) \le N^{2s-3 \left(1-\frac{4}{s+3}\right)+o(1)} \qquad \text{and} \qquad I_{s, d}^{\sharp}(\delta;N) \le N^{2s-s(d) \left(1-\frac{d^2}{2s+d^2+1}\right)+o(1)},
$$ 
respectively.
Clearly, $\delta \to 1$ as $s \to \infty$, so it is no surprise that these expressions converge to the bound of~\eqref{eq:MVT} (and thus also Conjecture~\ref{conj:Wooley}) as $s$ tends to infinity. 

Our upper bounds are complemented by the following general lower bounds.

\begin{theorem}\label{thm:lower d=2} 
	Fix $s>0$.  
	\begin{enumerate}
		\item \label{it:lower d=2 delta>>N}
		If $\delta\ge c_1/N$ for some absolute constant  $c_1>0$, we have 
		$$
			I_{s, 2}^{\flat}(\delta;N) \gg \delta^{2} N^{s-1} \max\left\{1, \(\delta N\)^{s-2}\right\}.
		$$
	
		\item 
		If  $\delta\ge c_2/\sqrt{N}$ for some absolute constant $c_2>0$, we have	
		$$
			I^{\flat}_{s,2} (\delta;  N) \gg \delta^{2} N^{3(s-1)/2}.
		$$
	\end{enumerate}
\end{theorem}

We observe that for $\delta\ge c_2/\sqrt{N}$ the second bound of Theorem~\ref{thm:lower d=2}  improves the first bound, which at the point $\delta=N^{-1/2}$ takes the form $\delta^{2}N^{3s/2-2}$.

Our methods also give a bound for dimension $d \ge 2$. For $1\le k<d$, it is convenient to define 
\begin{equation}
\label{eq:nu d k}
\nu(d, k)=  \min\left\{ \frac{1}{2k} ,  \frac{1}{2d-k} \right \}.
\end{equation}  
In that notation, our bound is as follows.

\begin{theorem}
\label{thm:lower-d} 
Fix any $s>0$ and $k \in \{1, \ldots, d\}$. For any $\delta$ with $\delta\ge C N^{-\nu(d, k)} \log N$ for some sufficiently large constant $C$, we have 
$$
 I_{s, d}^{\flat}(\delta;N) \ge
    \delta^d  N^{d+s-s(d)+o(1)}  \max\left \{  1, \(\delta ^{1/ \nu(d, k)} N\)^{s-d}\right\}.
$$
\end{theorem}

In particular, for $s \le d$ the bound of Theorem~\ref{thm:lower-d} simplifies as 
$$
 I_{s, d}^{\flat}(\delta;N)  \ge
    \delta^d  N^{s+d-s(d)+o(1)}
$$
which does not depend on $k$, and thus holds for $\delta \ge N^{-\mu(d)}$
where 
$$
\mu(d) = \max_{k=1, \ldots, d} \nu(d, k). 
$$
We obviously have 
$$
	 \mu(d) \sim \frac{3}{4d} \qquad (d \to \infty).
$$
Moreover, a list of explicit values of $\mu(d)$ for $2 \le d \le 10$ is given in Table~\ref{tab: mu}. 

\begin{table}[H]
	\centering
	\begin{tabular}{|c||c|c|c|c|c|c|c|c|c|}
		\hline
		$d$ & $2$  & $3$ & $4$ & $5$ & $6$ & $7$ & $8$ & $9$ & 10 \\ \hline
		$\mu(d) $ & $1/3$  & $1/4$ & $1/6$ & $1/7$ & $1/8$ & $1/10$ & $1/11$ & $1/12$ & $1/14$\\
		\hline
	\end{tabular} 
	\caption{Values of $\mu(d)$ for $d =2, \ldots, 10$}
	\label{tab: mu}
\end{table}

\subsection{Discussion and comparison of our results} 
Here we compare the bounds proposed by Demeter and Langowski~\cite[Conjecture~1.3]{DeLa} as well as Wooley~\cite[Conjecture~8.2]{Wool3} with our Conjecture~\ref{conj1} as well as with our other upper bounds. It should be emphasised that we do this in the case of  $s =2,3$ for which~\cite[Conjecture~1.3]{DeLa} is actually established in~\cite[Theorem~2.4]{DeLa}.

To compare our various upper bounds, it is convenient to define
\begin{align*}
& \kappa_{s,d}^{(0)} (\tau)= \limsup_{N\to\infty} \sup_{\|\va\|_\infty\le 1} \frac{\log I^{(0)}_{s,d}(N^{-\tau}; \va, N)}{\log N},\\
& \kappa_{s,d}^{\sharp} (\tau)= \limsup_{N\to\infty} \frac{\log  I^{\sharp}_{s,d}(N^{-\tau};N)}{\log N}, 
\end{align*}
where in $\kappa_{s,d}^{(0)} (\tau)$, the inner supremum is taken over all sequences of complex weights
with $\|\va\|_\infty\le 1$. 
It follows from~\eqref{eq:I and I} that 
$$
\kappa_{s,d}^{\sharp} (\tau) \le  \kappa_{s,d}^{(0)} (\tau). 
$$
We now present some plots of $\kappa_{s,d}^{\sharp} (\tau)$
and $\kappa_{s,d}^{(0)} (\tau)$ for small values of $d$ and $s$,  which help to 
compare various bounds and conjectures.

\begin{figure}[h]
	\begin{tikzpicture}
	\begin{axis}[
	xticklabels={$0$,$ $,$1$,$ $,$2$,$ $},
	yticklabels={${-1}$,$ $,$0$,$ $,$1$,$ $,$2$},
	axis x line=middle,
	axis y line=left,
	height=8cm,
	ytick pos=left,
	ytick={-1,-0.5,...,2.5},
	xtick={0,0.5,...,2.5},
	minor y tick num=1,
	minor x tick num=1,
	y label style={rotate=-90},
	xlabel={{$\tau$}},
	ymin=-1.5, ymax=2.5,
	xmin=0, xmax=2.5,
	legend style={
		cells={anchor=west},
		legend pos=outer north east,
	}
	]
	
	\addplot[solid, thick] expression[domain=0:2] {2-x};
	\addlegendentry{\scriptsize{Conj.~\ref{conj2}, $ \kappa_{2,2}^{(0)}$}}
	
	\addplot[loosely dashdotted, thick] expression[domain=0:3] {3-(3/2)*x};
	\addlegendentry{\scriptsize{D--L~\cite{DeLa},  $\kappa_{2,2}^{(0)}$}}	
	
	\addplot[densely dotted, thick]expression[domain=0:0.5] {2-2*x};
	\addplot[densely dotted, thick, forget plot]expression[domain=0.5:1] {1};   
	% need the "forget plot" so there's only one legend entry for both functions
	\addplot[densely dotted,thick, forget plot]expression[domain=1:2] {2-x};
	\addplot[densely dotted,thick, forget plot]expression[domain=2:3] {4-2*x};
	\addlegendentry{\scriptsize{Cor.~\ref{cor:mvt-weight-small d}, $ \kappa_{2,2}^{(0)}$}}
		
	\end{axis}
	\end{tikzpicture}
	\caption{Comparison of upper bounds and conjectures on $\kappa_{2,2}^{(0)}(\tau)$ and 	
	$\kappa_{2,2}^{\sharp}(\tau)$ for various values of $\delta=N^{-\tau}$. Wooley's conjecture (Conjecture~\ref{conj:Wooley}) is identical to our Corollary~\ref{cor:mvt-weight-small d}, but applies only in the range $\tau \le 1/4$.}
	\label{fig:I22}
\end{figure}
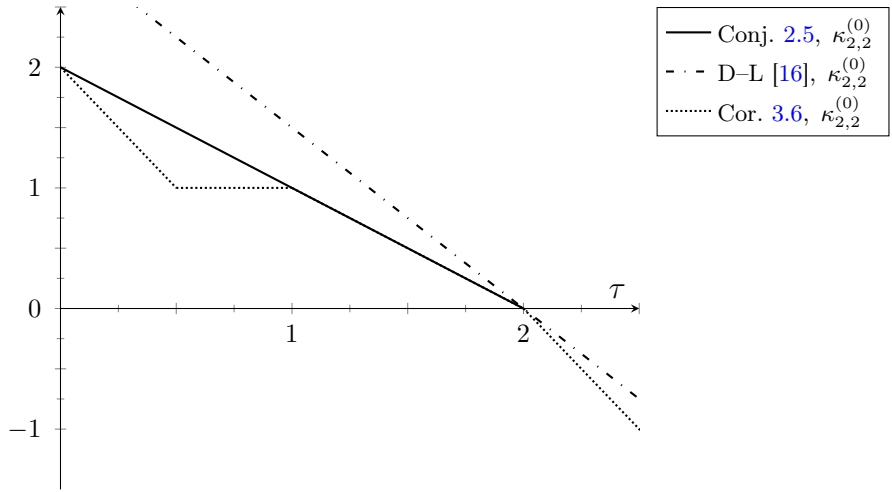

Figure~\ref{fig:I22} compares the bounds proposed by Demeter and  Langowski~\cite[Theorem~2.4]{DeLa} and Wooley~\cite[Conjecture~8.2]{Wool3}, as well as the upper bound of Corollary~\ref{cor:mvt-weight-small d} and the lower bounds of Theorem~\ref{thm:lower d=2}, in the case $d=s=2$. We note that the results and conjectures of~\cite{Wool3} apply only to $I_{s,d}^{\sharp}(\delta; N)$, while ours apply to the more general quantity $I^{(0)}_{s,d}(\delta; \va, N)$ for $\|a\|_{\infty} \le 1$. 

Observe that Demeter and Langowski~\cite{DeLa} conjecture (and prove) diagonal behaviour up to the point $s=\rho(2)/2=1$, which puts our configuration of parameters into the supercritical range. In contrast, our more flexible formulation in Conjecture~\ref{conj2} allows us to choose parameters in such a way that the value $s=2$ does correspond to the critical point. Indeed, from~\eqref{eq:alpha-bound} we see that the choice of $\alpha=1$ is optimal for our choice of parameters, and consequently our conjecture takes a stronger form than the result obtained by Demeter and Langowski~\cite{DeLa}.
Moreover, it is evident that at least for the choice of parameters at hand, our conjecture is fully established by the bounds of Corollary~\ref{cor:mvt-weight-small d}. We also note that our Corollary~\ref{cor:mvt-weight-small d} coincides with the bound conjectured by Wooley~\cite[Conjecture~8.2]{Wool3} in the latter one's range of applicability, but is valid for a significantly larger range of $\delta$.

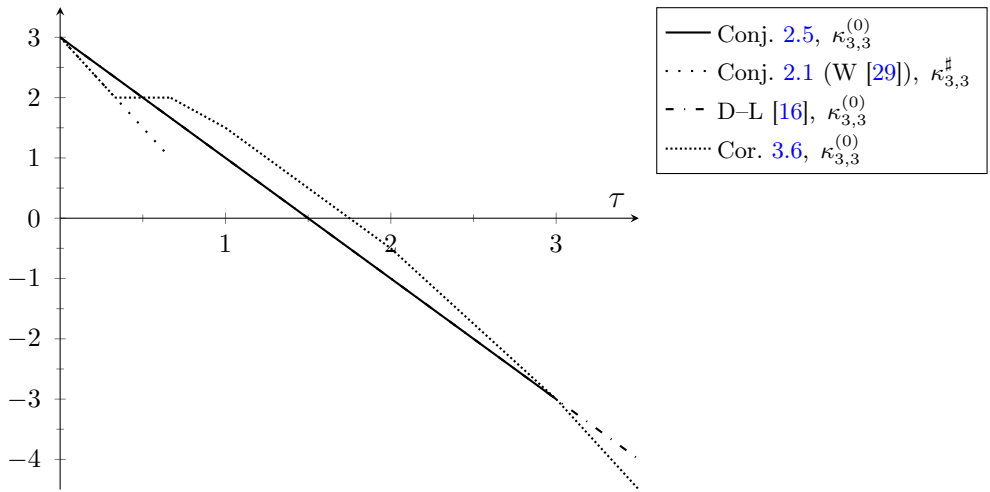
\begin{figure}[h]	
	\begin{tikzpicture}
	\begin{axis}[
	xticklabels={$0$,$1$,$2$,$3$},
	yticklabels={$-4$,${-3}$,$ {-2}$,$ {-1}$,$0$,$1$,$ 2$,$ 3$},
	axis x line=middle,
	axis y line=left,
	height=8cm,
	ytick pos=left,
	ytick={-4,-3,...,3},
	xtick={0,1,...,3},
	minor y tick num=1,
	minor x tick num=1,
	y label style={rotate=-90},
	xlabel={{$\tau$}},
	ymin=-4.5, ymax=3.5,
	xmin=0, xmax=3.5,
	legend style={
		cells={anchor=west},
		legend pos=outer north east,
	}
	]
	
	\addplot[solid, thick] expression[domain=0:3] {3-2*x};
	\addlegendentry{\scriptsize{Conj.~\ref{conj2},  $\kappa_{3,3}^{(0)}$}}
	
	\addplot[loosely dotted,thick] expression[domain=0:0.67] {3-3*x};
	\addlegendentry{\scriptsize{Conj.~\ref{conj:Wooley} (W~\cite{Wool3}), $\kappa_{3,3}^{\sharp}$}}
	
	\addplot[loosely dashdotted, thick] expression[domain=0:4] {3-2*x};
	\addlegendentry{\scriptsize{D--L~\cite{DeLa},  $\kappa_{3,3}^{(0)}$}}	
	
	\addplot[densely dotted, thick]expression[domain=0:0.33] {3-3*x};
	\addplot[densely dotted, thick, forget plot]expression[domain=0.33:0.67] {2};  
	\addplot[densely dotted, thick]expression[domain=0.67:1] {3-(3/2)*x}; 
	% need the "forget plot" so there's only one legend entry for both functions
	\addplot[densely dotted,thick, forget plot]expression[domain=1:2] {7/2-2*x};
	\addplot[densely dotted,thick, forget plot]expression[domain=2:3] {9/2-(5/2)*x};
	\addplot[densely dotted,thick, forget plot]expression[domain=3:4] {6-3*x};
	\addlegendentry{\scriptsize{Cor.~\ref{cor:mvt-weight-small d}, $ \kappa_{3,3}^{(0)}$}}

	\end{axis}
	\end{tikzpicture}
	\caption{Comparison of upper bounds and conjectures on $\kappa_{3,3}^{(0)}(\tau)$ and 	
	$\kappa_{3,3}^{\sharp}(\tau)$    for various values of $\delta=N^{-\tau}$. Observe that in this situation, the bounds of our Conjecture~\ref{conj2} and the result of Demeter and Langowski~\cite{DeLa} coincide. Wooley's conjecture~\cite{Wool3}  applies to $\tau \le 2/3$.}\label{fig:I33}
\end{figure}

In Figure~\ref{fig:I33} we present the proved and conjectured bounds for $\kappa^{(0)}_{3,3}(\tau)$ and $\kappa^\sharp_{3,3}(\tau)$. In this setting, Demeter and Langowski~\cite[Conjecture~1.3]{DeLa} address the case $\alpha=(d-1)/2=1$, so in view of~\eqref{eq:alpha-bound} the critical point of their conjecture coincides with that of our Conjecture~\ref{conj2}, and consequently they anticipate the same bound.

Our Corollary~\ref{cor:mvt-weight-small d} gives bounds which are actually stronger than that in~\cite[Conjecture~1.3]{DeLa} and Conjecture~\ref{conj2} for $\delta>N^{-1/2}$, but is not strong enough to establish them in the full range. It also establishes with Wooley's conjecture~\cite[Conjecture~8.2]{Wool3} for $\delta \ge N^{-1/3}$.
Note that for $\delta<N^{-3}$ the trivial  bound~\eqref{eq: Triv Bound zero} is sharp.

\begin{figure}[h]
	
	\begin{tikzpicture}
	\begin{axis}[
	xticklabels={$0$,$1$,$2$,$3$},
	yticklabels={${-6}$,$ {-5}$,$ {-4}$,$ {-3}$,$ {-2}$,$ {-1}$,$0$,$1$,$ 2$,$ 3$},
	axis x line=middle,
	axis y line=left,
	height=8cm,
	ytick pos=left,
	ytick={-6,-5,...,4},
	xtick={0,1,...,3},
	minor y tick num=1,
	minor x tick num=1,
	y label style={rotate=-90},
	xlabel={{$\tau$}},
	ymin=-6.5, ymax=3.5,
	xmin=0, xmax=3.5,
	legend style={
		cells={anchor=west},
		legend pos=outer north east,
	}
	]
	
	\addplot[solid,thick] expression[domain=0:3] {2-(7/3)*x};
	\addlegendentry{\scriptsize{Conj.~\ref{conj2},  $\kappa_{2,3}^{(0)}$}}
	
	\addplot[loosely dotted, thick] expression[domain=0:0.67] {2-3*x};
	\addlegendentry{\scriptsize{Conj.~\ref{conj:Wooley} (W~\cite{Wool3}), $\kappa_{2,3}^{\sharp}$}}	
	
	\addplot[loosely dashdotted, thick] expression[domain=0:4] {2-2*x};
	\addlegendentry{\scriptsize{D--L~\cite{DeLa},  $\kappa_{2,3}^{(0)}$}}	
	
	\addplot[densely dotted, thick]expression[domain=0:0.33] {2-3*x};
	\addplot[densely dotted,thick, forget plot]expression[domain=0.33:0.67] {1};   
	% need the "forget plot" so there's only one legend entry for both functions
	\addplot[densely dotted,thick, forget plot]expression[domain=0.67:1] {2-(3/2)*x};
	\addplot[densely dotted,thick, forget plot]expression[domain=1:1.5] {5/2-2*x};
	\addplot[densely dotted, thick, forget plot]expression[domain=1.5:4] {4-3*x};
	\addlegendentry{\scriptsize{Cor.~\ref{cor:mvt-weight-small d}, $ \kappa_{2,3}^{(0)}$}}

	\end{axis}
	\end{tikzpicture}
	\caption{Comparison of upper bounds and conjectures on $\kappa_{2,3}^{(0)}(\tau)$ and 	
	$\kappa_{2,3}^{\sharp}(\tau)$    for various values of $\delta=N^{-\tau}$. Wooley's conjecture~\cite{Wool3} applies to $\tau \le 2/3$.}\label{fig:I32}
\end{figure}
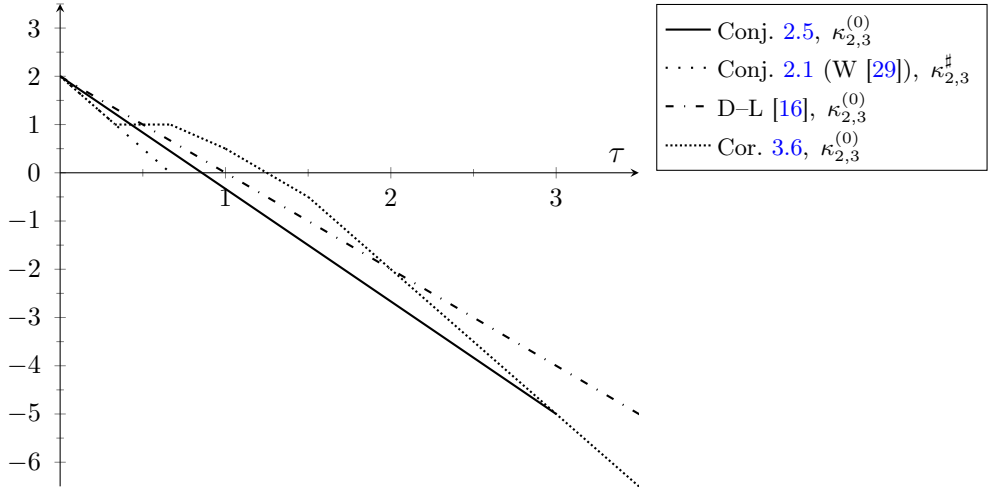

Our final Figure~\ref{fig:I32} compares the bounds for $\kappa^{(0)}_{3,2}(\tau)$ and $\kappa^{\sharp}_{3,2}(\tau)$. Again, it is obvious from the graph that the theorem by Demeter and Langowski, optimised for a different set of parameters, fails to be sharp in this setting, and indeed, we obtain sharper bounds in our Corollary~\ref{cor:mvt-weight-small d} for all $\delta<N^{-2}$ as well as $\delta>N^{-1/2}$. In our Conjecture~\ref{conj1}, we are allowed to take $\alpha<1$, and it follows from~\eqref{eq:alpha-bound} that the value $\alpha=2/3$ is optimal. As in the previous setting, this conjecture is overfulfilled for $\delta>N^{-3/7}$, but open for $N^{-3/7}>\delta>N^{-3}$. We see again that our bounds establish Wooley's conjecture~\cite[Conjecture~8.2]{Wool3} for $\delta \ge N^{-1/3}$, but fall short in the range $N^{-2/3}<\delta<N^{-1/3}$.

\begin{remark}
	A common feature of Figures~\ref{fig:I22},~\ref{fig:I32} and~\ref{fig:I33} that the bounds in the extreme ranges $\tau>d$ (corresponding to $\delta \le N^{-d}$) and $\tau <1/d$ (corresponding to $\delta > N^{-1/d}$) are represented by non-coinciding parallel lines. This is particularly intriguing since in both of these ranges the bounds are proven to be sharp, which raises the question of what the `truth' looks like between these two ranges. Our result of Corollary~\ref{cor:mvt-weight-small d} shows that the `true' graph cannot be entirely convex or entirely concave, even in the otherwise well-understood case of small degrees and few variables. Instead, there we detect a noticeable plateau at the peak at the origin, and a lowland plain for the averages over larger boxes, but the shape of the slope connecting the two is unclear. This is an indication that the average behaviour of exponential sums over short intervals (and by extension their pointwise behaviour) is governed by phenomena that are poorly understood and deserving of more investigation. 
\end{remark}

\begin{remark}
	We omitted to include our lower bounds in the graphs. The reason for this is that since our lower bounds are uniform in $\bxi$, that is, the location of the box within the unit torus. In contrast, our upper bounds either specifically discuss or at least accommodate the box located at the origin, where the exponential sum is known to have a spike. Thus, the lower bounds are of no representative value in the vicinity of the origin, where our upper bounds are known to be sharp. We have no evidence whether the lower bound might be sharp at some $\bxi$ away from the origin. 
\end{remark}

\section{Proof of Theorem~\ref{thm:conj}} 
Let $\delta \in [N^{-d}, 1]$ be fixed, and define $\mathcal D = [-\delta, \delta]^d$. Write further 
$$
	\mathcal C = \prod_{j=1}^d [-cN^{-j},cN^{-j}]
$$ 
for some positive  $c<1/(8d)$. Clearly, for $\vx \in \cC$ we have $|x_1n+\ldots+x_dn^d|\le 1/8$ and hence 
$$
 |S_d(\vx;N)| \gg N.
$$

Define $\kappa \in [0,d]$ by the relation $\delta^{-1} = N^{\kappa}$, and put $k=\lfloor\kappa\rfloor$ and $\tau=\kappa-k=\{\kappa\}$, so that we have the inequalities $N^{-(k+1)}<\delta\le N^{-k}$. 
Since 
$$
\vol( \mathcal C \cap \mathcal D) \asymp \delta^k  \prod_{j=k+1}^d N^{-j} \asymp (N^{-(k+\tau)})^k N^{-s(d) +s(k)} \asymp N^{ -s(d)-k(k-1+2\tau)/2},
$$
where by convention the empty product is taken to have value 1,
we have 
$$
\int_{\mathcal D} |S_d(\vx;N)|^{2s} dx\gg \vol (\mathcal C \cap \mathcal D ) N^{2s} \asymp N^{2s - s(d)-k(k-1+2\tau)/2}.
$$

From the definition of $s_0(d, \alpha)$ we also have the requirement that
$$
I_{s,d}^{(0)}(\delta;N) \le \delta^{d-\alpha}N^{s+o(1)}
$$ 
for $s \le s_0(d,\alpha)$. Thus we need that 
$$
N^{2s - s(d) -k(k-1+2\tau)/2} \le (N^{-k-\tau})^{d-\alpha}N^{s+o(1)},
$$
and in particular 
$$
s \le \frac{d(d+1)}{2}+\frac{k(k-1+2\tau)}{2}-(k+\tau)(d-\alpha).
$$

Recall now that we aim for a statement that holds for all $\delta \in [N^{-d},1]$. We therefore want to minimise the expression 
$$
	F(k,\tau) =\frac{d(d+1)}{2}+\frac{k(k-1+2\tau)}{2}-(k+\tau)(d-\alpha).
$$
Observe that formally we have 
\begin{equation}\label{eq: F continuity}
	F(k,1)=F(k+1,0),
\end{equation} as can be confirmed by a straightforward computation. Thus, we can extend the range of $\tau \in [0,1)$ by including the endpoint. 

Suppose first that $\alpha \not\in \Z$. Clearly, we have 
\begin{equation}\label{eq: d F/d tau}
	\partial F(k,\tau) / \partial \tau = k-(d-\alpha).
\end{equation} 
Consequently, for any fixed value of $k$ the function $F(k,\tau)$ is minimal for $\tau=0$ when $k>d-\alpha$, and for $\tau=1$ when $k<d-\alpha$. 

Assume first that $k>d-\alpha$, so that we can assume that $\tau=0$. In this case we have 
$$
	\frac{\partial F(k,0)}{\partial k} = k-(d-\alpha+1/2),
$$
which is optimal when $k$ is taken to be the integer that is closest to $d-\alpha+1/2$. Upon writing $d-\alpha+1/2 = \lfloor d-\alpha \rfloor + 1+(\{d-\alpha\}-1/2)$ and observing that $(\{d-\alpha\}-1/2)\in (-1/2, 1/2)$, we see that this closest integer is given by $k=\lfloor d-\alpha \rfloor+1$.

Similarly, if $k<d-\alpha$, we have $\tau=1$ and thus 
$$
\frac{\partial F(k,1)}{\partial k} = k-(d-\alpha-1/2).
$$
In this case we have $d-\alpha-1/2 = \lfloor d-\alpha \rfloor + (\{d-\alpha\}-1/2)$, where we note that $(\{d-\alpha\}-1/2)\in (-1/2, 1/2)$, so that the optimal value for $k$ in this setting is given by $k=\lfloor d-\alpha \rfloor$. 

Consequently, the function $F(k,\tau)$ is minimised by either $k=\lfloor d-\alpha \rfloor+1$ and $\tau=0$, or for $k=\lfloor d-\alpha \rfloor$ and $\tau=1$. Upon recalling~\eqref{eq: F continuity}, it is clear that these values coincide. 
It thus remains to compute the value of the minimum by inserting the values $k=\lfloor d-\alpha \rfloor+1$ and $\tau=0$. Upon writing $\lfloor d-\alpha \rfloor=d-\alpha - \{ d-\alpha \}$ and noting that $\{d-\alpha\}=1-\{\alpha\}$ we find that 
\begin{align*}
	s &\le F(\lfloor d-\alpha \rfloor+1,0)\\
	&= \frac{d(d+1)}2 + \frac{\lfloor d-\alpha \rfloor(\lfloor d-\alpha \rfloor+1)}{2} - (\lfloor d-\alpha \rfloor+1)( d-\alpha)\\
	&= \frac{d(d+1)}2 - \frac{(d-\alpha )(d-\alpha+1)}{2} - \frac{\{d-\alpha\}}{2} + \frac{(\{d-\alpha\})^2}{2}\\
	&= \frac{\alpha(2d-\alpha+1)}{2} - \frac{\{\alpha\}(1-\{\alpha\})}{2}.
\end{align*}

Finally, when $\alpha \in \Z$ we conclude from~\eqref{eq: d F/d tau} that $F(k,\tau)$ is minimal for $\tau=0$ when $k>d-\alpha$ and for $\tau=1$ when $k<d-\alpha$, and that it is constant in $\tau$ when $k=d-\alpha$. In combination with the continuity property~\eqref{eq: F continuity} it follows that $F$ is minimised for $k=d-\alpha$ and on the entire interval $\tau \in [0,1]$, and we have the explicit value
$$F(d-\alpha+1,0)= \frac{d(d+1)}{2} - \frac{(d-\alpha)(d-\alpha+1)}{2} = \frac{\alpha(2d-\alpha+1)}{2}
$$ 
as well. This completes the proof of Theorem~\ref{thm:conj}.

\begin{remark}
It remains to comment on the situation when $\delta \le N^{-d}$. Indeed, adapting the strategy of the above proof to this eventuality, we find that $\vol(\mathcal C \cap \mathcal D) \asymp \delta^d$, and consequently $\delta^{d}N^{2s} \ll I^{(0)}_{s,d}(\delta;N)$, which matches the trivial bound~\eqref{eq: Triv Bound zero}. 
\end{remark}

\section{Transition to inhomogeneous mean values} 
 \label{sec:inhom}

In the following, we denote by  $\sfJ_{s,d} (\delta; N)$ the number of solutions to the system of $d$ inequalities
$$
\left|\sum_{j=1}^{2s} (-1)^j n_j^i\right|\le \delta^{-1}  , 
\qquad i =1, \ldots, d,
$$
in integer  variables $1\le  n_1,\ldots,n_{2s} \le N$.
We recall~\cite[Lemma~3.8]{CKMS}, in a form which is better suited for our applications. 

\begin{lemma}
	\label{lem:I and J}
	If $|a_n| \le 1$, $n =1, \ldots, N$, then 
	$$
	I_{s, d}^{(0)}(\delta;\va, N)  \le \delta^{d}  \sfJ_{s,d} (\delta; N).
	$$
\end{lemma}

Recall the definition of $J_{s,d} (\vh; N)$ from the preamble of~\eqref{eq:Triv Jh} above. The following is~\cite[Theorem~1.3]{Wool3}.

\begin{lemma}
	\label{lem:Inhom VMVT}
	Suppose that $d \ge 2$ and $\vh\ne \mathbf{0}$. Let $\ell$ be the smallest integer for which $h_{\ell} \ne 0$, and suppose that $\ell \le d-1$. Then for any integer $s \le d(d + 1)/2$, we have
	$$
		J_{s,d} (\vh; N) \le  N^{s- 1/2 + o(1)} +  N^{s-\eta_{s,d}(\ell) +o(1)} , 
	$$
		where $\eta_{s,d}(j)$ is as given in~\eqref{eq:def-eta}.
\end{lemma}
We point out that we do not have any bound in the situation when $\ell=d$. Observe moreover that $J_{s,d} (\vh; N)=0$ trivially when $|h_j| > 2sN^j$ for any $j = 1, \ldots, d$.

Define 
$$
\cU = [-\delta^{-1}, \delta^{-1}]^d \subseteq \Z^d \mand \cV = \prod_{j=1}^d [-2sN^j, 2sN^j] \subseteq \Z^d.
$$ 
Then for $1 \le j \le d$ put
$$
	\cH_j = \{\vh \in \cU \cap \cV:~\vh = (0,\ldots,0,h_j, \ldots, h_d), \; h_j \neq 0\}.
$$
In this notation, we have the obvious partition 
$$
	\cU \cap \cV = \{\bm 0\} \cup \bigcup_{j=1}^d \cH_j,
$$
so that 
\begin{equation}\label{eq:mv-partition}
	\sfJ_{s,d} (\delta; N) = \sum_{\vh \in \cU \cap \cV} J_{s,d} (\vh; N) 
= J_{s,d} (N) + \sum_{j=1}^d \sum_{\vh \in \cH_j} J_{s,d} (\vh; N).
\end{equation}

Next we note that for each $j=1, \ldots d$ we have 
$$
	\# \cH_j \asymp\prod_{i=j}^d\min\{\delta^{-1}, N^i\} = \delta^{-(d-j+1)} \prod_{i=j}^d\min\{1, N^i \delta\}.
$$
In particular, if   $\delta \in [N^{-k-1} ,N^{-k})$ with some integer $k$ we have 
$$
	\# \cH_j \asymp \delta^{-(d-j+1)} \prod_{\substack{i=j \\ i \le k}}^d (N^i \delta),
$$
where the empty product should be interpreted as having value $1$. Consequently, 
we may write 
\begin{equation}\label{eq: Hj-size}	
	\# \cH_j \asymp \begin{cases}
		\delta^{-(d-j+1)} & \text{for } k < j,\\
		\delta^{-(d-k)}N^{(k(k+1)-j(j-1))/2} &  \text{for } j \le k < d,  \\
		N^{(d(d+1)-j(j-1))/2} &  \text{for } d \le  k.
	\end{cases}
\end{equation}
For future reference we also record the obvious fact that 
$$
	\# \cH_1 \ge \ldots \ge \# \cH_d,
$$ 
as well as the bound
\begin{equation}\label{eq: H1-size}
	\# \cH_1 \asymp \delta^{-d+k}N^{s(k)}
\end{equation}
which is valid for $0 \le k \le d$.

Finally, we also record the following simple bound.

\begin{lemma}
	\label{lem:Inhom VMVT-Cauchy}
	Suppose that $d \ge 2$. For any finite set $\cH \subseteq \Z^d$ we have
	$$
		\sum_{\vh \in \cH} J_{s,d} (\vh; N) \le \(\# \cH J_{2s,d} (N)\)^{1/2}. 
	$$ 
\end{lemma}

\begin{proof}
	By Cauchy's inequality we have
	\begin{align*}
		\(\sum_{\vh \in \cH} J_{s,d} (\vh; N)   \)^2 & \le  \# \cH  \sum_{\vh \in \cH} J_{s,d} (\vh; N)  ^2\\
		& \le \# \cH  \sum_{\vh \in \Z^d} J_{s,d} (\vh; N)  ^2 =  \# \cH   J_{2s,d} (N),
	\end{align*}
	and the result follows. 
\end{proof}

\section{Proof of Theorems~\ref{thm: mv weights small s}--~\ref{thm:mvt-weight-conj}}

\subsection{General upper bounds for weighted Weyl sums over small boxes}

We are now ready to establish our most general upper bound for Weyl sums, which implies Theorems~\ref{thm: mv weights small s} and~\ref{thm: mv weights large s} as well as Theorem~\ref{thm:mvt-weight-conj} as special cases. The following result serves as a starting point for all ensuing deliberations. 

\begin{prop}\label{prop:mvt-weight} Suppose that $\| \va \|_{\infty} \le 1$.
	\begin{enumerate}
	\item \label{prop:mvt-weight-uncond}
	For any $s\le s(d)$ we have 
	\begin{align*}
		I_{s, d}^{(0)}(\delta;\va, N)  \le \delta^d & N^{s+o(1)}  \biggl(\min\left\{\#\cH_d, (\#\cH_d)^{1/2} \(1+N^{s-s(d)/2}\)\right\}   \\
		&\quad   +  \sum_{j=1}^{d-1}  \min \left\{\#\cH_j( N^{-1/2}+ N^{-\eta_{s,d}(j)}), (\#\cH_j)^{1/2} (1+N^{s-s(d)/2}) \right\}\biggr),
	\end{align*}
	where $\eta_{s,d}(j)$ is as given in~\eqref{eq:def-eta}.
	
	\item \label{prop:mvt-weight-cond}
	Suppose now that~\eqref{eq:Conj Jh} is known for some $\nu$ and some $s_1(d)$. 	For all integers $s \le s_1(d)$ and any $\delta \in (0 ,1]$ we have the potentially stronger bound
$$
		I_{s, d}^{(0)}(\delta;\va, N)  \le \delta^d N^{s+o(1)} \left(1+  \min\{\#\cH_1N^{-\nu}, (\#\cH_1)^{1/2}(1+N^{s-s(d)/2})\}\right).
$$
	\end{enumerate}
\end{prop} 

\begin{proof}
	Our starting point is the decomposition~\eqref{eq:mv-partition}.  First we observe that we can apply Lemma~\ref{lem:Inhom VMVT} on the first $d-1$ of the inner summands. Furthermore, for $j=d$ we use  the bound 
	$$
		\sum_{\vh \in \cH_d} J_{s,d}(\vh;N) \ll \min\left\{\#\cH_d J_{s,d}(N), (\#\cH_d)^{-1/2} J_{2s,d}(N)^{1/2}\right\},
	$$ 
	which combines~\eqref{eq:Triv Jh} and Lemma~\ref{lem:Inhom VMVT-Cauchy}. Recalling~\eqref{eq:MVT}, we obtain 
	$$
		\sum_{\vh \in \cH_d}  J_{s,d}(\vh;N) \ll N^{s+o(1)}\min\left\{\#\cH_d, (\#\cH_d)^{1/2} \(1+N^{s-s(d)/2}\)\right\} .
	$$ 
	Similarly, for $1 \le j \le d-1$ we have 
	$$
		\sum_{\vh \in \cH_j}  J_{s,d}(\vh;N) \ll N^{s+o(1)}\min\left\{\#\cH_j (N^{-1/2} + N^{-\eta_{s,d}(j)}), (\#\cH_j)^{1/2} \(1+N^{s-s(d)/2}\)\right\} .
	$$ 
	Combining both of these bounds with the result of Lemma~\ref{lem:I and J} leads to the desired conclusion in the unconditional case. 
	
	For the conditional setting we only need to make some minor modifications to the above argument. Again starting from~\eqref{eq:mv-partition}, we can now use~\eqref{eq:Conj Jh} inside all of the inner summands. Thus, for $1 \le j \le d$ we have 
	$$
	\sum_{\vh \in \cH_j}  J_{s,d}(\vh;N) \ll N^{s+o(1)} \min\left\{\#\cH_j N^{-\nu}, (\#\cH_j)^{-1/2} (1+N^{s-s(d)/2})\right\}.
	$$ 
	Substituting this back into~\eqref{eq:mv-partition} and invoking Lemma~\ref{lem:I and J} yields 
	\begin{align*} 
		I_{s, d}^{(0)}(\delta;\va, N) &\le \delta^dN^{s+o(1)} \left( 1 + \sum_{j=1}^{d}\min\left\{\#\cH_j N^{-\nu}, (\#\cH_j)^{1/2}(1+N^{s-s(d)/2}) \right\}  \right) \\
		&\le N^{s+o(1)}\left(1+  \min\left\{\#\cH_1 N^{-\nu}, (\#\cH_1)^{1/2}(1+N^{s-s(d)/2}) \right\}\right),
	\end{align*}  
	where in the last step we have used that $\# \cH_1= \max_j \# \cH_j$.
\end{proof}

\subsection{Proofs of Theorems~\ref{thm: mv weights small s} and~\ref{thm:mvt-weight-conj}}
We now specialise to the case $s \le s(d)/2$. In that situation, the conclusion of Proposition~\ref{prop:mvt-weight}\eqref{prop:mvt-weight-uncond} can be simplified significantly. 
\begin{lemma}
	For any integer $s\le s(d)/2$ and any $\delta \in (0 ,1]$  we have 
	\begin{align*}
		I_{s, d}^{(0)}(\delta;\va, N)  \le \delta^d N^{s+o(1)}  \left(\min\{\#\cH_1N^{-1/2}, (\#\cH_1)^{1/2}\} + (\#\cH_d)^{1/2}\right).
	\end{align*}
\end{lemma} 

\begin{proof} 
	Recall Proposition~\ref{prop:mvt-weight}\eqref{prop:mvt-weight-uncond}. Clearly, under the assumptions of the lemma we have $N^{s-s(d)/2} \ll 1$. Moreover, for $s$ in the admissible range we have 
	$$
		\min_{1 \le j \le d-1}\eta_{s,d}(j) \ge \frac{s(d)}{2} \cdot \min_{1 \le j \le d-1} \frac{d-j+1}{d+j+1} \ge \frac{s(d)}{2} \cdot \frac{2}{2d} = \frac{d+1}{4} > 1/2
	$$
	for all $d \ge 2$.
	Consequently, the conclusion of Proposition~\ref{prop:mvt-weight} simplifies to  
	\begin{align*} 
		I_{s, d}^{(0)}(\delta;\va, N) &\le \delta^dN^{s+o(1)} \left( 1 + \sum_{j=1}^{d-1}\min\left\{\#\cH_j N^{-1/2}, (\#\cH_j)^{1/2} \right\} + (\#\cH_d)^{1/2} \right) \\
		&\le N^{s+o(1)}  \left(  \min\left\{\#\cH_1 N^{-1/2}, (\#\cH_1)^{1/2} \right\} +  (\#\cH_d)^{1/2}\right), 
	\end{align*} 
			where in the last step we  used that $\# \cH_1= \max_j \# \cH_j$. 
	This concludes the proof.
\end{proof} 
The derivation of Theorems~\ref{thm: mv weights small s} and~\ref{thm:mvt-weight-conj} is now straightforward. We note from~\eqref{eq: Hj-size} that $\#\cH_1 \gg N$ for all $\delta<N^{-1/d}$. This is obvious for $\delta<N^{-1}$, and can be checked in a straightforward manner for $N^{-1}<\delta<N^{-1/d}$. In those situations, we have 
$$
	\min\{\#\cH_1N^{-1/2}, (\#\cH_1)^{1/2}\} = (\#\cH_1)^{1/2}  \ge (\#\cH_d)^{1/2},
$$
and the bound becomes 
$$
	I_{s, d}^{(0)}(\delta;\va, N) \le \delta^d N^{s+o(1)} (\#\cH_1)^{1/2}.
$$
Finally, if $\delta>N^{-1/d}$ we see from~\eqref{eq: Hj-size} that
$$
	\min\{\#\cH_1N^{-1/2}, (\#\cH_1)^{1/2}\} = \#\cH_1N^{-1/2}\asymp \delta^{-d}N^{-1/2}, 
$$
so that we obtain
$$
	I_{s, d}^{(0)}(\delta;\va, N) \le \delta^d N^{s+o(1)} (\delta^{-d}N^{-1/2} + \delta^{-1/2}).
$$ 
When $\delta>N^{-1/(2d-1)}$, the first term prevails. The proof of Theorem~\ref{thm: mv weights small s} is complete upon combining both of these bounds with~\eqref{eq: Hj-size}.

We now pivot to the proof of Theorem~\ref{thm:mvt-weight-conj}, where we suppose that~\eqref{eq:Conj Jh} is known with $\nu=1$. At this point, the bound in part~\eqref{thm-item:mvt-weight-conj large s} of the theorem is immediate from Proposition~\ref{prop:mvt-weight}\eqref{prop:mvt-weight-cond} upon inserting~\eqref{eq: H1-size}. 

To establish the bounds of part~\eqref{thm-item:mvt-weight-conj small s}, we begin by noting that $\#\cH_1 \gg N^2$ for all $\delta<N^{-2/d}$. This is again immediate from~\eqref{eq: H1-size} for $\delta\le N^{-2}$ and straightforward to check in the intervals $N^{-2} < \delta \le N^{-1}$ and $N^{-1}<\delta \le N^{-2/d}$, respectively. Consequently, in this range of $\delta$ we find that 
$$
	\min\left\{\#\cH_1 N^{-1}, (\#\cH_1)^{1/2} \right\} = \#\cH_1^{1/2} > 1.
$$ 
Finally, for $N^{-2/d}<\delta<1$ we obtain
$$
	\min\left\{\#\cH_1 N^{-1}, (\#\cH_1)^{1/2} \right\} = \#\cH_1 N^{-1} \begin{cases}\ll 1 & \text{if }N^{-2/d} < \delta < N^{-1/d}  \\  \gg 1 & \text{if } N^{-1/d} \le \delta\le 1. \end{cases}
$$
The conclusion of Theorem~\ref{thm:mvt-weight-conj}\eqref{thm-item:mvt-weight-conj small s} is now complete upon using these bounds within Proposition~\ref{prop:mvt-weight}\eqref{prop:mvt-weight-cond} and inserting the values of~\eqref{eq: H1-size}.

\subsection{Proofs of Theorem~\ref{thm: mv weights large s} and Corollary~\ref{cor: mvt weights large s and delta}}
We now investigate the situation when $s(d)/2 < s < s(d)$ and $\delta>N^{-1}$. In that situation we have $\#\cH_j \asymp \delta^{-d+j-1}$ for $1 \le j \le d$. Moreover, since $\delta \ge N^{-1} \ge N^{s(d)-2s}$, we clearly have 
$$
\min\left\{\delta^{-1}, \delta^{-1/2} N^{s-s(d)/2}\right\}  = \delta^{-1}.
$$
Thus, under these conditions the conclusion of Proposition~\ref{prop:mvt-weight} reads
\begin{align*} 
	I_{s, d}^{(0)}(\delta;\va, N)  \le \delta^d N^{s+o(1)}  \biggl( \delta^{-1} +  \sum_{j=1}^{d-1} \min\bigl\{ \delta^{-d+j-1}(N^{-1/2} &+ N^{-\eta_{s,d}(j)}), \\
	&\qquad \delta^{-(d-j+1)/2}N^{s-s(d)/2}\bigr\}  \biggr).
\end{align*}
This completes the proof of Theorem~\ref{thm: mv weights large s}.

To finish the proof of Corollary~\ref{cor: mvt weights large s and delta}, we begin by noting that $\delta^{-1}< \delta^{-(d-1+j)/2}N^{s-s(d)/2}$ for all $j$ and all $\delta \le 1$. Consequently, it is sufficient to check in what range of $\delta$ one has 
$$
	\delta^{-1} \gg \max_{1 \le j \le d-1}\delta^{-d+j-1}(N^{-1/2} + N^{-\eta_{s,d}(j)}) \asymp \delta^{-d}N^{-1/2} + \max_{1 \le j \le d-1}\delta^{-d+j-1} N^{-\eta_{s,d}(j)}.
$$ 

On comparing these terms and recalling the definition of $\eta_{s,d}(j)$ from~\eqref{eq:def-eta} it is enough to choose 
$$
	\delta\ge \max\{N^{-(s(d) -s)  \vartheta(d) }, N^{-1/(2d-2)}\}, 
$$
with 
$$
 	\vartheta(d)  =  \min _{j=1, \ldots, d-1} f(j),
$$
where the function $f$ is defined by~\eqref{eq: func f}. The proof is thus complete if we can show that this definition of $\vartheta(d)$ coincides with the one given in~\eqref{eq: theta}. 

Since the denominator of $f(x)$ vanishes at $x=d$ and at $x=-d-1$, neither of which lie in the interval $[1,d-1]$, we see that $f$ is continuous inside said interval. Moreover, simple but somewhat tedious calculus shows that 
$$
	f'(x) =  -  \frac{d^2-1 - 2(d+1)x + x^2}{\(d^2+d-x^2-x\)^2}.
$$
This expression has two roots at $x_{\pm}=d+1\pm \sqrt{2(d+1)}$, of which we can disregard the larger one since it is clearly outside the interval $[1,d-1]$. Since $f'$ has a sign change from negative to positive at $x=x_-$, that root corresponds to a minimum. In order to compute the value, note that for $d \ge 3$ we have $d+1 - \sqrt{2(d+1)} \in [1,d-1]$, so that both $d+1-\lfloor \sqrt{2(d+1)}\rfloor$ and $d+1-\lceil \sqrt{2(d+1)}\rceil$ lie in the set $\{1,\ldots,d-1\}$. Thus, the values for $\vartheta(d)$ certainly coincide for $d \ge 3$. Finally, for $d=2$ the identity is straightforward to check explicitly. 
This completes the proof of Corollary~\ref{cor: mvt weights large s and delta}.

\section{Approach via the structure  of large Weyl sums} 
\label{sec:struct} 

In what follows it is be convenient to define 
$$
	D = \min\left\{2^{d-1}, 2d(d-1)\right\}. 
$$

We begin our analysis with a description of the structure of large Gauss sums
\begin{equation} 
\label{eq:GauusSum}
G(x_1, x_2;N) = S_2((x_1,x_2);N) = \sum_{n=1}^{N} \e\(x_1n +x_2n^2\).
\end{equation} 
The following is~\cite[Lemma~5.1]{BCS2}, which in turn follows from a result of Baker~\cite[Theorem~3]{Bak0} (see also~\cite[Theorem~4]{Bak1}).

\begin{lemma} 
\label{lem:structure of large Gauss} 
We fix some  $\varepsilon > 0$, and  suppose  that for a real $A> N^{1/2+ \varepsilon}$ we have $|G(x_1, x_2;N)|\ge A$ for some $(x_1,x_2)\in \R^2$. 
Then there exist integers $q, a_1, a_2$ such that 
$$
1 \le q \le  \(NA^{-1}\)^2 N^{o(1)},
$$
and for $i=1, 2$ we have 
$$
\left| x_i - \frac{a_i}{q} \right | \le  (NA^{-1})^2 q^{-1}  N^{-i + o(1)}.
$$
\end{lemma} 

For $d \ge 3$,  we use the following result from~\cite{BCS1} which is based on a combination  of results of Baker~\cite[Theorem~3]{Bak0} and~\cite[Theorem~4]{Bak1}  with bounds of complete rational sums, see, for example,~\cite{CPR}. 
Namely by~\cite[Lemma 2.7]{BCS1} we have

\begin{lemma}
\label{lem:Struct Large Weyl d>2}
We fix $d\ge 3$, some  $\varepsilon > 0$, and suppose  that for a real number $A$ satisfying $A> N^{1-1/D + \varepsilon}$ we have $|S_{d}(\vx; N)| \ge A$ for some $\mathbf{x} \in \Tor$. 
Then there exist positive integers $q_2, \ldots, q_d$  with  $\gcd(q_i,q_j) = 1$ for $2 \le i < j \le d$, such that  
\begin{itemize} 
	\item[(i)] $q_2$ is cube-free,
	\item[(ii)]  $q_i$ is $i$-th power-full but $(i+1)$-th power-free when  $3\le i\le d-1$,
	\item[(iii)]  $q_d$ is $d$-th power-full, 
\end{itemize}
and
$$
 \prod_{i=2}^d q_i^{1/i} \le N^{1+o(1)}A^{-1}, 
$$
and integers $b_1, \ldots, b_d$ with 
$$
\gcd\(q_2\cdots q_d, b_1, \ldots, b_d\)=1  
$$ 
such that 
$$
\left|x_j-\frac{b_j}{q_2\cdots q_d}\right |\le (NA^{-1})^{d} N^{-j+o(1)}  \prod_{i=2}^d q_i^{-d/i} , \qquad j=1, \ldots, d.
$$
\end{lemma} 

\begin{remark} \label{rem: good approx} 
For errors of  the approximations  
to $x_1, x_2$ of Lemma~\ref{lem:structure of large Gauss}, by the condition of  $A>N^{1/2+\varepsilon}$ we have 
\begin{equation} 
\label{eq:err1}
(NA^{-1})^2 q^{-1}  N^{-i + o(1)} \le q^{-1} N^{-2 \varepsilon +o(1)} , \qquad i =1,2.
\end{equation}  
Similarly, for errors of~Lemma~\ref{lem:Struct Large Weyl d>2} we have 
\begin{equation} 
\label{eq:err2}
(NA^{-1})^{d} N^{-j+o(1)}  \prod_{i=2}^d q_i^{-d/i} \le  N^{-d \varepsilon+o(1)} \prod_{i=2}^d q_i^{-1}   , \qquad j=1, \ldots, d.
\end{equation}
\end{remark} 

For a  real  $A>0$, we define the level set.  
\begin{equation} 
\label{eq:FdA}
\sF_{d, A}=\{\vx\in \T_d:~|S_d(\vx; N)|\ge A\}.
\end{equation}
Further, for a box $\fB(\bxi, \delta)=\bxi+ [0,\delta]^d\subseteq \T_d$, denote 
\begin{equation} 
\label{eq:lamnbda_dAN}
\lambda_{d,  \bxi} (\delta, A;N) =\lambda(\fB(\bxi, \delta)\cap \sF_{d, A}).
\end{equation}

\begin{lemma}  
\label{lem:level set for large A d=2}
Suppose that $ A> N^{1/2 + \varepsilon}$ for some fixed $\varepsilon>0$. Then for any $\delta\ge AN^{-1}$ we have 
$$
\lambda_{2,  \bxi} (\delta, A;N)\le  \delta^{2}  N^{3 + o(1)} A^{-6}. 
$$
\end{lemma}

\begin{proof}
 Let $Q=(NA^{-1})^{2}N^{\eta}$ for some small $\eta>0$. For $q\in \N$ and $\mathbf b =(b_1, b_2)\in [q]^{2}$ define the rectangular box
$$
R_{q}(\mathbf b)= B(b_1/q, Q q^{-1}N^{-1}) \times B(b_2/q, Qq^{-1}N^{-2}),
$$
where $B(x, r)\subseteq \R$ denotes the interval with center $x$ and radius $r$. Clearly, each such box has area 
$$
 \lambda(R_{q}(\mathbf b)) \asymp Q^2/(q^2N^3).
$$

By Lemma~\ref{lem:structure of large Gauss}, for all sufficiently large $N$ we obtain 
$$
\sF_{2,A}\subseteq \bigcup_{q\le Q} \bigcup_{(b_1, b_2)\in [q]^{2}} R_{q}(\mathbf b).
$$

It is an easy consequence of~\eqref{eq:err1} that the boxes $R_{q}(\mathbf b)$ are disjoint for all $q \in \N$. It follows that any box $\fB(\bxi, \delta)$ 
intersects with at most $O\(1+(q\delta)^2\)$ boxes $ R_{q}(\mathbf b)$. Consequently, 
recalling~\eqref{eq:lamnbda_dAN}, we derive 
\begin{align*}
\lambda_{2,  \bxi} (\delta, A;N) & =
\lambda(\fB(\bxi, \delta)\cap \sF_{2,A})\\
& \ll \sum_{q \le Q} \sum_{(b_1,b_2) \in [q]^2} \lambda\(R_{q}(\mathbf b) \cap \fB(\bxi, \delta)\) \\
&\ll \sum_{q \le Q} \sum_{\substack{\mathbf b \in [q]^2 \\ R_{q}(\mathbf b) \cap \fB(\bxi, \delta) \neq \emptyset}} \lambda\(R_{q}(\mathbf b)\)\\
&\ll \sum_{q \le Q} \(1+(q\delta)^2\) \frac{Q^2}{q^2N^3} \ll  \frac{Q^{2}}{N^{3}} + \frac{Q^{3}}{N^{3}} \delta^2.  
\end{align*}
Therefore,  using $\delta\ge AN^{-1} \ge Q^{-1/2}$, we derive  
$$
\lambda_{2,  \bxi} (\delta, A;N) \ll \frac{Q^{2}}{N^{3}} + \frac{Q^{3}}{N^{3}} \delta^2\ll  \frac{Q^{3}}{N^{3}} \delta^2=N^{3+3\eta}A^{-6}\delta^{2}.
$$
Since  $\eta>0$ is arbitrary we obtain the desired bound. 
\end{proof}

For $d\ge 3$, we mimic the proof of~\cite[Lemma~2.9]{BCS1} in order to obtain a level set estimation with restriction to some small box. Formally, taking $k=d$ in~\cite[Lemma~2.9]{BCS1} and adding a factor of $\delta^{d}$ there, we have the following bound. 

\begin{lemma} 
\label{lem:level set for large A d>2}
Suppose that $d\ge 3$ and $A> N^{1-1/D + \varepsilon}
$ for some fixed $\varepsilon>0$. Then for any $\delta\ge \(AN^{-1}\)^{1/d}$ we have 
$$
\lambda_{d,  \bxi} (\delta, A;N)\le  N^{d^2+1-s(d)+o(1)}A^{-d^2-1}  \delta^{d}. 
$$  
\end{lemma}

\begin{proof} 
Let 
\begin{equation}\label{eq: def Q}
Q=\(NA^{-1}\)^{d}N^{\eta}
\end{equation} for some small number $\eta>0$. For any $q_2,\ldots,q_d\in \N$ and $b_1, \ldots, b_d \in \Z$, define the box 
$$
	R_{q_2,\ldots,q_d}(\mathbf b) =  \biggl\{\vx \in \Tor:~ \left|x_j-\frac{b_j}{q_2 \ldots q_d}\right |\le Q N^{-j}  \prod_{i=2}^d q_i^{-d/i}, \ j=1, \ldots, d\biggr\}.
$$
Again, we note that 
\begin{equation}\label{eq: vol R}
\lambda(R_{q_2,\ldots,q_d}(\mathbf b)) \asymp  \prod_{j=1}^d\left(QN^{-j} \prod_{i=2}^d q_i^{-d/i} \right) \asymp Q^dN^{-s(d)}\prod_{i=2}^d q_i^{-d^2/i}.
\end{equation}
Moreover, by~\eqref{eq:err2} these boxes are pairwise disjoint. Thus, for any fixed $d$-tuple $\(q_2,\ldots,q_d\)$, the number of boxes $R_{q_2,\ldots,q_d}(\mathbf b)$ intersecting $\fB(\bxi,\delta)$ nontrivially is given by 
\begin{equation}
\label{eq:intersect-bound2}
\# \{\mathbf b \in [q_2 \cdots q_d]^d:~R_{q_2, \ldots, q_d}(\mathbf b) \cap \fB(\bxi, \delta) \neq \emptyset\} = 
O(1+(\delta q_2\cdots q_d)^{d}).
\end{equation}

For any integer $i\ge 2$ it is convenient to denote   
$$
\cF_{i}=\{n \in \N:~  \text{$n$ is $i$-th power full}\} \quad \text{and} \quad
\cF_{i}(x)= \cF_i\cap [1,x],
$$ 
so that an easy counting shows that
\begin{equation}\label{eq: powerful bound}
	\# \cF_{i}(x) \ll x^{1/i},
\end{equation}
and to put 
$$
	\Omega =\left\{\(q_2, \ldots, q_d\)\in \N^{d-1}:~ q_i \in \cF_i \quad (3 \le i \le d), \quad  \prod_{i=2}^d q_i^{d/i} \le  Q  \right\}.
$$ 
Thus, recalling the definition~\eqref{eq:FdA},
we clearly have 
$$
	\sF_{d,A}\subseteq \bigcup_{(q_2, \ldots, q_d)\in \Omega} \bigcup_{\mathbf b \in [q_2 \cdots q_d]^d} R_{q_2, \ldots, q_d}(\mathbf b).
$$
Combining this with~\eqref{eq: vol R} and~\eqref{eq:intersect-bound2}, 
and recalling~\eqref{eq:lamnbda_dAN}, we can estimate 
\begin{equation}\label{eq: partition1}
	\begin{split} \lambda_{d,  \bxi} (\delta, A;N) &= \lambda \(\fB(\bxi,\delta) \cap \sF_{d,A}\)\\
	&\le \sum_{(q_2,\ldots, q_d)\in \Omega} \sum_{\substack{\mathbf b \in [q_2 \cdots q_d]^d\\R_{q_2,\ldots, q_d}(\mathbf b) \cap \fB(\bxi, \delta) \neq \emptyset}} \lambda(R_{q_2,\ldots, q_d}(\mathbf b))\\
	&\ll \sum_{(q_2,\ldots, q_d)\in \Omega} (1+(\delta q_2 \cdots q_d)^d) Q^d N^{-s(d)} \prod_{i=2}^d q_i^{-d^2/i}.
	\end{split} 
\end{equation} 
Write 
$$
	U_1 = \sum_{(q_2,\ldots, q_d)\in \Omega}\prod_{i=2}^{d}q_i^{-d^{2}/i} \qquad \text{and} \qquad  U_2 = \sum_{(q_2,\ldots, q_d)\in \Omega}\prod_{i=2}^{d}q_i^{d-d^{2}/i},
$$
then~\eqref{eq: partition1} can be bounded by
\begin{equation}\label{eq: partition}
	\lambda \(\fB(\bxi,\delta) \cap \sF_{d,A}\) \ll Q^{d}N^{-s(d) } U_1 + \delta^{d} Q^{d}N^{-s(d) } U_2.
\end{equation}

Clearly, 
\begin{equation} \label{eq:U1-bound}
	U_1 \le  \sum_{(q_2,\ldots, q_d)\in \N^d} \prod_{i=2}^{d}q_i^{-d^{2}/i} \ll 1.
\end{equation}
We now turn to the estimation of $U_2$. For $Q_2, \ldots, Q_d$ and $\delta>0$ denote 
$$
	\Omega(Q_2,\ldots,Q_d)=\{(q_2, \ldots, q_d) \in \Omega:~ Q_i/2 < q_i \le Q_i, \ i =2, \ldots, d\},
$$
and write 
$$
	U_2(Q_2,\ldots,Q_d) = \sum_{(q_2,\ldots, q_d)\in \Omega(Q_2,\ldots,Q_d)} \prod_{i=2}^{d}q_i^{d-d^{2}/i}.
$$
Thus, covering  $\Omega$ by $O\(\(\log N\)^d\)$ dyadic boxes, we see that 
\begin{equation} \label{eq:Dyadic Uqq}
	U_2
	\ll \max \biggl \{U_2\(Q_2,  \ldots, Q_d\):~ Q_2, \ldots, Q_d \ge 1, \ \prod_{i=2}^d Q_i^{-d/i} \le  Q \biggr\} \(\log N\)^d.
\end{equation}
By~\eqref{eq: powerful bound}, this yields
\begin{equation}\label{eq:U2-bound}
	U_2\(Q_2,  \ldots, Q_d\) \ll   \sum_{Q_2/2 <q_2 \le Q_2} q_2^{d-d^2/2} \prod_{i=3}^{d}\left(Q_i^{d-d^{2}/i} \# \cF_i\(Q_i\) \right) \ll \prod_{i=2}^{d}Q_i^{\alpha_i},
\end{equation}
where 
$$
	\alpha_2=d-d^2/2+1 \mand  \alpha_i= d-(d^2-1)/i \quad (i=3, \ldots, d).
$$
Observe that for every $i=2, \ldots, d$ we have $\alpha_i\le 1/i$. Combining this with the condition on $Q_2, \ldots, Q_d$ in~\eqref{eq:Dyadic Uqq}, we derive from~\eqref{eq:U2-bound} and~\eqref{eq:Dyadic Uqq} that 
\begin{equation}\label{eq:U2-bound final}
U_2 \ll Q^{1/d} (\log N)^d.
\end{equation}

Finally, we can combine the bounds of~\eqref{eq: partition}, \eqref{eq:U1-bound} and~\eqref{eq:U2-bound final}. Thus, and recalling the condition $\delta\ge (AN^{-1})^{1/d}$ as well as the definition of $Q$ from~\eqref{eq: def Q} together with arbitrary choice of $\eta>0$, we obtain
\begin{align*}
\lambda (\fB(\bxi,\delta) \cap \sF_{d,A})&\ll Q^{d}N^{-s(d)}+Q^{d+1/d}N^{-s(d)}\delta^{d}(\log N)^d\\
&\le (NA^{-1})^{d^{2}+1}N^{-s(d)+o(1)}\delta^{d},
\end{align*}
which finishes the proof.
\end{proof}

\section{Proofs of Theorems~\ref{thm:mvt d=2} and~\ref{thm:mvt d>2}}

\subsection{Proof of Theorem~\ref{thm:mvt d=2}}

Let $\bxi\in \T_2$ and 
$$
	A= N^{1/2 + s/(6+2s)}.
$$ 
Next, we divide the set $\fB(\bxi, \delta)=\bxi+[0, \delta]^{2}$ into two parts depending on whether $|S_2(\vx; N)|\ge A$ or not. Thus  combining with Lemma~\ref{lem:level set for large A d=2}, which applies for 
the above choice of $A$,  we derive 
\begin{align*}
 	I_{s, 2}^{\sharp} (\delta;  N) &\le A^{2s}\delta^{2}+N^{2s} \sup_{\bxi \in \T_2} \lambda(\{\vx\in \fB(\bxi, \delta):~|S_2(\vx; N)|\ge A\})\\
	&\le A^{2s}\delta^{2} +N^{2s+3+o(1)}A^{-6} \delta^{2},
\end{align*}
which   yields the desired bound.

\subsection{Proof of Theorem~\ref{thm:mvt d>2}}

Let $\bxi\in \T_d$ and 
$$
	A= N^{1-s(d)/(2s+d^2+1)},
$$ 
noting that the hypothesis $s>(s(d)D-d^2-1)/2$ ensures that $A>N^{1-1/D+\eps}$, so that Lemma~\ref{lem:level set for large A d>2} is applicable. 
Divide the box $\fB(\bxi, \delta)=\bxi+[0, \delta]^{d}$ into two parts depending on whether $|S_d(\vx; N)|\ge A$ or not. Thus, applying Lemma~\ref{lem:level set for large A d>2}  we obtain 
\begin{align*}
 	I_{s, d}^{\sharp}(\delta;N)&\le A^{2s}\delta^{d}+N^{2s} \sup_{\bxi \in \T_d} \lambda (\{\vx\in \fB(\bxi, \delta):~|S_d(\vx; N)|\ge A\})\\
	&\le A^{2s} \delta^{d}+N^{2s +d^2+1-s(d)+o(1)}A^{-d^2-1}  \delta^{d},
\end{align*} 
which   yields the desired bound.

\section{Rational exponential sums}
\subsection{Gauss sums} 

Recall the definition~\eqref{eq:GauusSum} of Gauss sums. We also record their explicit evaluation, which is classical (see, for example,~\cite[Equation~(1.55)]{IwKow}). 
\begin{lemma}
\label{lem:Gauss}
Let $p\ge 3$ be a prime number and $a, b\in \F_p$ with $b\neq 0$, then 
$$
\left|\sum_{n=0}^{p-1}\ep\left(an+bn^{2}\right)\right|=p^{1/2}.
$$
\end{lemma} 

We also recall a classical result of  Fiedler, Jurkat and K\"orner~\cite[Lemma~4]{FJK}.
\begin{lemma}
\label{lem:IncompGauss}
For any  prime  $p$ and any $a, b \in \F_p$ with $b\neq 0$ we have 
$$
\max_{1\le M, N\le p}\left|\sum_{M+1\le n\le M+N}\ep\left(an+bn^{2}\right)\right| \ll p^{1/2}. 
$$
\end{lemma}

\begin{lemma}
\label{lem:Cont Gauss}
Let $p$ be a prime and $N$ an integer with $N \ge Cp$ for some positive constant $C$. Suppose that the pair $(x_1, x_2) \in \T_2$ has a rational approximation of the shape 
$$
|x_1-a/p|\le c/N \mand |x_2-b/p|\le c /N^{2} 
$$
for some positive constant $c$, where $\gcd(b,p) =1$. Then we have 
$$
|G(x_1, x_2; N)|\gg N p^{-1/2}.
$$
\end{lemma}

\begin{proof}  
Combining Lemma~\ref{lem:IncompGauss} with~\cite[Corollary~2.6]{ChSh-JMAA} we obtain a continuity property of Gauss sums.
$$
G(x_1, x_2; N)-G(a/p, b/p; N) \ll Np^{-1/2}\(|x_1-a/p|N+|x_2-b/p|N^{2}\).
$$
Since by  Lemmas~\ref{lem:Gauss} and~\ref{lem:IncompGauss} 
we have 
$$
|G(a/p, b/p; N)| = \fl{N/p} p^{1/2} + O\(p^{1/2}\) =  Np^{-1/2} + O\(p^{1/2}\),
$$
for an appropriate choice of $C$ we obtain the desired  result. 
\end{proof} 

\subsection{Rational sums with arbitrary polynomials} 
 
For $d\ge 3$ we do not have an analogue of Lemma~\ref{lem:Gauss}. 
For an arbitrary box $\bxi+ [0,\delta]^d\in \Tor$, we follow the same strategy as in~\cite{ChSh-AM} on the distribution of large complete rational sums. In fact, we need a more refined version of the argument presented in~\cite[Lemma~2.6]{ChSh-AM} that provides quantitative estimates on the number of large sums inside any given small box. The,n using a method similar to those employed in the treatment of the case $d=2$, we obtain some nontrivial lower bounds.
 
Let $p$ be a prime  and let $\F_p$ denote the finite field of $p$ elements. For  a vector $\vu = (u_1, \ldots, u_d)\in \F_p^d$ we consider the rational exponential sum
$$
T_{d, p} (\vu) =S_d(\vu/p;p)=\sum_{n=1}^{p}\ep\(u_1 n+\ldots +u_d n^{d} \),
$$
where $\ep(z) = \e(z/p)$. We also consider discrete cubic boxes 
\begin{equation}
\label{eq:box}
\fB=\cI_1\times \ldots \times \cI_d \subseteq \F_p^{d}
\end{equation}
with side-length $L$, where for each $j=1,\ldots,d$ the set $\cI_j = \{k_j+1, \ldots, k_j +L\}$ is a set of $L \leqslant p $ consecutive integers,
reduced modulo $p$ if $k_j +L \ge p$.

Our goal is to establish a quantitive version of~\cite[Lemma~2.6]{ChSh-AM}. As in~\cite{ChSh-AM} we start with recalling that by a result of Knizhnerman and  Sokolinskii~\cite[Theorem~1]{KnSok1} (see also~\cite{KnSok2}) we have the following. 
\begin{lemma}
\label{lem:LB Dens}
For every integer $d \geqslant 2$ there are some positive constants $c_d$ and $\gamma_d$ having the property that there exists a set $\cL_p \subseteq \F_p^d$ of cardinality  $\#\cL_p \geqslant c_dp^d$ such that for all $\va \in \cL_p$ one has
$$
\left| T_{d, p} (\va)\right|  \geqslant \gamma_d \sqrt{p}.
$$ 
\end{lemma}

We also need a result on the distribution of monomial curves. The following is~\cite[Lemma~2.5]{ChSh-AM}, which we augment by also including the (trivial) case $k=1$. 

\begin{lemma}  \label{lem:k}
	Let $(a_1, \ldots, a_k)\in (\F_p^*)^{k}$. Then  there exists a positive constant $C$ which depends only on $k$ such that for any box $\fB$ as in~\eqref{eq:box} with sidelength $L\geqslant C p^{1-1/2k} \log p$ for $k \ge 2$ and  $L \ge 1$ for $k=1$,  we have 
	$$
 		\#\left\{\lambda\in \F_p^{*}:~(a_1 \lambda , \ldots, a_k \lambda^{k})\in \fB \right\} \geqslant  \frac12 L^{k}p^{1-k}. 
	$$   
\end{lemma}

We are now ready to establish our main result of this section. Recall the definition of $\nu(d,k)$ from~\eqref{eq:nu d k}, then we have the following level-set result.
 
\begin{lemma}\label{lem:box dense} 
	For any $d \ge 3$ and $1\le k<d$ there exist  constants $\gamma_d, \Gamma_d>0$, such that for any box $\fB$ as in~\eqref{eq:box} with side-length $L\geqslant \Gamma_d p^{1- \nu(d, k)} \log p$,  we have
	$$
		\# \left\{\vu \in \fB:~ \left |T_{d, p} (\vu) \right|\ge \gamma_d p^{1/2}  \right\}  \gg L^d. 
	$$   
\end{lemma}

\begin{proof}
	Adjusting $\Gamma_d $ if necessary, we can assume that $p$ is large enough. By Lemma~\ref{lem:LB Dens}, there is a constant $\gamma_d$ and a set $\cL_p \subseteq \F_p^d$ of cardinality 
	\begin{equation}\label{eq: Lp lower bound}
		\# \cL_p \ge c_p p^d
	\end{equation}
	for some suitable constant $c_p$, and having the property that $T_{d,p}(\va) \ge \gamma_d \sqrt p$ for all elements $\va \in \cL_p$.  Clearly, if $(a_1, \ldots, a_d) \in \cL_p$, then for any $\lambda\in \F_p^*$ we also have $(a_1\lambda, \ldots, a_d \lambda^d) \in \cL_p$. 

	Denote by $\cA_k\subseteq  \F_p^{k}$ the set of all $(a_1,\ldots, a_k) \in \F_p^{k}$ for which 
	\begin{equation}\label{eq: def A_k}
		\# \( \cL_p \cap \left((a_1,\ldots a_k) \times \F_p^{d-k}\right)\) \ge  \frac12  c_d p^{d-k},
	\end{equation}
	where $c_d$ is the constant of Lemma~\ref{lem:LB Dens}. 
	Then by decomposing $\F^k = \cA_k \cup (\F_p^k \setminus \cA)$ and using~\eqref{eq: def A_k} (in the contrapositive form) within the second term, we have 
	\begin{equation}\label{eq: Lp upper bound} \begin{split}
		\# \cL_p &= \sum_{(a_1,\ldots,a_k) \in \cA_k} \sum_{\substack{(a_{k+1},\ldots,a_d) \in \F_p^{d-k} \\ (a_1,\ldots,a_k) \in \cL_p}} 1 + \sum_{(a_1,\ldots,a_k) \in \F_p^k \setminus \cA_k} \sum_{\substack{(a_{k+1},\ldots,a_d) \in \F_p^{d-k} \\ (a_1,\ldots,a_d) \in \cL_p}} 1 \\
		&\le \#\cA_kp^{d-k} + \sum_{(a_1,\ldots,a_k) \in \F_p^k \setminus \cA_k} \frac12  c_d p^{d-k}  \\
		&\le  \#\cA_kp^{d-k} +\frac12  c_d p^{d}  .
		\end{split}
	\end{equation} 
	On combining the bounds~\eqref{eq: Lp lower bound} and~\eqref{eq: Lp upper bound}, we find that 
$$
	c_d p^d \le  \#\cA_kp^{d-k} +  \frac12  c_d p^{d}
$$
which rearranges to 
$$
	\#\cA_k \ge \frac{c_d}{2} p^k.
$$
Put now $\cA_k^* = \cA_k \cap (\F_p^*)^k$. Thus we clearly have 
\begin{equation}\label{eq:Ak}
	\#\cA_k^* \gg p^k.
\end{equation}

We now fix  $\va^*= (a_1, \ldots, a_k) \in \cA_k^*$ and consider the set
$$
	\cL_{p, k} \(\va^*\) =  \cL_p \cap \(\left\{a_1, \ldots, a_k\right\}\times \F_p^{d-k}\). 
$$
Clearly, from the definition~\eqref{eq: def A_k} of the set $\cA_k$ we have 
\begin{equation}\label{eq: cL_{p,k} lower bound}
	\# \cL_{p,k} \(\va^*\) \gg p^{d-k}.
\end{equation}

Given a box $\fB \subseteq \F_p^{d}$ of the form~\eqref{eq:box},  we decompose it in a natural way as $\fB=\fB_1\times \fB_2\subseteq \F_p^{k}\times \F_p^{d-k}$.
Note that we have  $\#\fB_1=L^{k}$. Let further 
$$
	\Lambda_k\(\va^*\) =\{\lambda\in \F_p^{*}:~(\lambda a_1, \ldots, \lambda^{k}a_k)\in \fB_1\}.
$$
Then Lemma~\ref{lem:k} implies that 
\begin{equation}
\label{eq:Lk}
\# \Lambda_k\(\va^*\) \geqslant \frac12 L^{k}p^{1-k},
\end{equation}
provided that the condition
\begin{equation}
\label{eq:L 1/2k}
L\geqslant \Gamma_d p^{1-1/2k} \log p
\end{equation}
is satisfied with a sufficiently large $\Gamma_d$ if $k\ge 2$, or for any $L$ if $k=1$. 

Let $R\(\va^*\)$ be the number of vectors of the form
$$\(\va^*, a_{k+1},\ldots a_d, \lambda\) = 
\(a_1, \ldots, a_k, a_{k+1},\ldots a_d, \lambda\)\in \cL_{p, k}\(\va^*\)\times \Lambda_k
$$ 
such that 
$$
(\lambda^{k+1}a_{k+1}, \ldots, \lambda^{d} a_d) \in \fB_2.
$$
It is shown in the proof of~\cite[Lemma~2.6]{ChSh-AM} that 
\begin{equation}\label{eq: R(a) bound} 
\begin{split}
\left| R\(\va^*\) - \# \cL_{p, k}\(\va^*\) \# \Lambda_k\(\va^*\) ( L/p)^{d-k}\right|&\\
  \le 
C_d \# \cL_{p, k}\(\va^*\)&(\# \Lambda_k(\va^*))^{1/2} (\log p)^{d-k}
\end{split} 
\end{equation} 
for some constant $C_d$ depending only on $d$. Suppose now that 
\begin{equation}\label{eq: L condition}
	C_d (\log p)^{d-k} \le \frac12(L/p)^{d-k} (\# \Lambda_k\(\va^*\) )^{1/2}.  
\end{equation}
Then the quantity $R(\va^*)$ from~\eqref{eq: R(a) bound} can be bounded below by 
\begin{equation}
\label{eq:Ra LowerBound}
\begin{split}
R\(\va^*\)  &  \ge  \frac12  \# \cL_{p, k}\(\va^*\)\#\Lambda_k\(\va^*\) ( L/p)^{d-k}\\
& \gg   p^{d-k}  L^{k}p^{1-k} ( L/p)^{d-k} \gg L^d p^{1-k},
\end{split} 
\end{equation}
where we used~\eqref{eq: cL_{p,k} lower bound} and~\eqref{eq:Lk}.

On the other hand,~\eqref{eq:Lk} implies that the condition~\eqref{eq: L condition} is certainly satisfied when  
$$
 \frac{1}{2\sqrt 2}(L/p)^{d-k} (L^{k}p^{1-k})^{1/2} \geqslant C_d (\log p)^{d-k},
$$ 
which can be rearranged to 
\begin{equation}
\label{eq:L 1/(2d-k)}
L \geqslant  \widetilde C_d p^{1-1/(2d-k)}(\log p)^{(d-k)/(d-k/2)},
\end{equation}
where $\widetilde C_d = (2 \sqrt 2 C_d)^{1/(d-k/2)}$. 
Note that since~\eqref{eq:Lk} is true for all $k \ge 1$, so is the last bound. 

Combining the conditions~\eqref{eq:L 1/2k} and~\eqref{eq:L 1/(2d-k)},  recalling  the definition of $\nu(d,k)$ in~\eqref{eq:nu d k} and increasing $\Gamma_d$   if necessary,  we see that the inequality  
$$L\geqslant \Gamma_d p^{1-\nu(d,k)}\log p$$ 
is sufficient  to guarantee that~\eqref{eq:Ra LowerBound}
holds for any $\va^* \in \cA_k^*$. 

Clearly, each vector of $\vu \in \F_p^d$ has at most $p$ representations as 
$$
\vu = (\lambda a_1, \ldots, \lambda^{d}a_d)$$  
with $(\a_1, \ldots,  a_d) \in \F_p^d$
and $\lambda \in \F_p^*$. 
Therefore, we derive from~\eqref{eq:Ra LowerBound} that  
$$
\# \left\{\vu \in \fB:~ \left |T_{d, p} (\vu) \right|\ge \gamma_d p^{1/2}  \right\}  \ge 
\frac{1}{p} \sum_{\va^* \in \cA_k^*} R\(\va^*\)   \gg  L^d p^{-k} \# \cA_k^*, 
$$
and recalling~\eqref{eq:Ak} we conclude the proof. 
\end{proof}

\subsection{Approximation of Weyl sums by rational sums}  
Let $\cZ_d$ be the set of vectors $\vu \in \F_p^d$ which are  not of the form $\vu= (u_1, 0, \ldots, 0)$. We also recall that the classical {\it Weil bound\/} (see, for example,~\cite[Chapter~6, Theorem~3]{Li} or~\cite[Theorem~5.38]{LN}), together with the completing technique described for instance in~\cite[Section~12.2]{IwKow}, implies that if $\vu \in\cZ_d$, then for any $N \le p$ we have
\begin{equation}\label{eq:IncompSum}
	\sum_{n=1}^{N}\ep\(u_1 n+\ldots +u_d n^{d} \) \ll p^{1/2} \log p.
\end{equation}

Using~\eqref{eq:IncompSum},  adapting the proof of~\cite[Lemma~2.9]{ChSh-AM} and noticing that the condition $p\mid N$ in~\cite[Lemma~2.9]{ChSh-AM} is not necessary  (see also~\cite[Corollary~2.6]{ChSh-JMAA}), we obtain the following continuity property for  Weyl sums. 

\begin{lemma}\label{lem:continuous}
	Let $\vu\in \F_p^{d}$ and  $\vx\in \T_d$, then we have 
	$$
		|S_d\(\vx; N\)-S_d\(p^{-1} \vu; N\)|\ll \frac{N\log p}{p^{1/2}} \sum_{j=1}^{n} \left |x_j-\frac{u_j}{p}\right |N^{j}.
	$$ 
\end{lemma}

Lemma~\ref{lem:continuous} immediately implies the following. 

\begin{lemma}\label{lem:Cont Weyl}
	Let $p$ be a prime, and let $\vu = (u_1, \ldots, u_d) \in \cZ_d$ such that 
	$$
		\left |T_{d, p} (\vu) \right|\ge \gamma_d p^{1/2}
	$$
	for some $\gamma_d>0$. 
	Then there are constants $c_d, C_d >0$ such that for all $N \ge C_d p$ and all  $\vx = (x_1, \ldots, x_d) \in \Tor$ satisfying 
	\begin{equation}\label{eq:approx log}
		\left |x_j-\frac{u_j}{p}\right | \le \frac{c_d}{N^j \log p}, \qquad j =1, \ldots, d,
	\end{equation}
	we have 
	$$
		|S_d(\vx; N)|\gg N p^{-1/2}.
	$$
\end{lemma}

\section{Proof of Theorems~\ref{thm:lower d=2} and~\ref{thm:lower-d}} 

\subsection{Proof of Theorem~\ref{thm:lower d=2}} 
Let $N\in \N$, and let and $c$ and $C$ be the constants of Lemma~\ref{lem:Cont Gauss}, noting that without loss of generality, we may assume that $c < C/2$. Suppose first that $\delta\ge 2C/N$, so that the interval $[\delta^{-1}, N/C]$ contains both the interval $[N/(2C),N/C]$ and $[\delta^{-1},2\delta^{-1}]$. Then for sufficiently large $N$ there is at least one prime number in the range 
\begin{equation}\label{eq: p range}
	N/C \ge p \ge  1/\delta.
\end{equation}

Now fix a point $\bxi \in \Tor$ and a $\delta>0$, and let $\widetilde R_{p}(\mathbf b)$
be the domain of admissible values of $(x_1, x_2) \in \fB(\bxi, \delta)$ having a rational approximation of the shape $|x_i - b_i/p| \le c N^{-i}$ for $i \in \{1,2\}$, where $p$ is a prime and $b_1, b_2 \in [p]$. This notation is reminiscent of that employed in our arguments in Section~\ref{sec:struct}, but we stress that we have different conditions imposed on on the exponential sums here than we had there. Write further
$$
	\fU_p(\bxi, \delta)=\bigcup_{\substack{\mathbf b\in [p]^{2}\\ \widetilde R_{p}(\mathbf b) \cap \fB(\bxi, \delta) \neq \emptyset}} \widetilde R_{p}(\mathbf b),
$$
noting that for all $p$ in the range~\eqref{eq: p range} we have $1/p > 2c/N$ and consequently the sets $\widetilde R_{p}(\mathbf b)$ are pairwise disjoint by our initial assumptions.
 
Since the number of pairs $\mathbf b \in [p]$ for which $\widetilde R_{p}(\mathbf b)$ intersects $\fB(\bxi,\delta)$ non-trivially is at least $(\delta p-1)^2 \ge (\delta p/4)^2$, and each individual box has volume $\lambda(\widetilde R_{p}(\mathbf b)) = (2c)^2N^{-3}$, it follows that 
$$
	\lambda(\fU_p(\bxi, \delta)) \ge ( c\delta p/2)^2N^{-3}. 
$$
Then, applying  Lemma~\ref{lem:Cont Gauss}, we derive 
\begin{align*}
	I_{s, 2}^{\flat} (\delta;  N)  &\ge \inf_{\bxi \in \T_2}  \left (\lambda(\fU_p(\bxi,\delta)) \inf_{(x_1, x_2) \in \fU_p(\bxi, \delta)} |G(x_1, x_2;N)|^{2s} \right) \\
	&\gg (\delta p)^{2} N^{-3} (Np^{-1/2} )^{2s}=\delta^{2} N^{2s-3}p^{2-s}.
\end{align*}

By the Prime Number Theorem,  for $s \le 2$ we can choose $p \in [N/(2C), N/C]$, while for 
$s >2$ we take $p \in [\delta^{-1}, 2 \delta^{-1}]$. 
Hence
$$
	I_{s, 2}^{\flat} (\delta;  N) \gg  \delta^{2} N^{s-1} \max\left\{1, \(\delta N\)^{s-2}\right\},
$$
which gives the desired lower bound in the case $\delta \gg N$. \medskip

To treat the case when $2C/ N \le \delta \le C'/ \sqrt N$, we first observe that for any distinct  fractions $a/q, b/r$ with coprime $q,r \in [\sqrt{N}, 2\sqrt{N}]$ we have 
$$
	\left |\frac{a}{q}-\frac{b}{r}\right |\ge \frac{1}{qr}\ge \frac{1}{N}.
$$
Thus for any distinct primes $p_1, p_2 \in [\sqrt{N}, 2\sqrt{N}]$ and for any $\bxi \in \T_2$ we have 
\begin{equation}\label{eq:disjoint}
	\fU_{p_1}(\bxi, \delta) \cap \fU_{p_2}(\bxi, \delta)=\emptyset,
\end{equation}
allowing us to enhance our previous arguments by summing over all primes in the interval $[\sqrt N, 2 \sqrt N]$. Then, proceeding in a similar way to before and applying Lemma~\ref{lem:Cont Gauss} and~\eqref{eq:disjoint}, we derive the lower bound
\begin{align*}
 	I_{s, 2}^{\flat}(\delta;  N) &\ge  \inf_{\bxi \in \T_2}\sum_{\substack{ \sqrt{N}\le p\le 2\sqrt{N}\\ p \text{ is prime }}} \int_{\fU_p(\bxi, \delta)} |G(x, y; N)|^{2s}dxdy\\
	& \gg \sum_{\substack{ \sqrt{N}\le p\le 2\sqrt{N}\\ p \text{ is prime }}}\left (\delta p\right )^{2} N^{-3} (Np^{-1/2} )^{2s}\\
	&\gg  \delta^{2} N^{3(s-1)/2}(\log N)^{-1},
\end{align*}
where the last inequality holds by the Prime Number Theorem. 

\subsection{Proof of Theorem~\ref{thm:lower-d}}

Recalling the definition~\eqref{eq:nu d k}, suppose that 
\begin{equation}\label{eq:delta large}
	\delta > 2\Gamma_d \log(N/C_d) \(\frac{N}{2C_d}\)^{-\nu(d,k)}
\end{equation} 
for some $k$, where $\Gamma_d$ and $C_d$ are the constants of Lemmas~\ref{lem:box dense} and~\ref{lem:Cont Weyl}, respectively. This choice of $\delta$ implies that the interval 
$$
	\left[\(2\Gamma_d \log \frac{N}{C_d}\)^{1/\nu(d,k)} \delta^{-1/\nu(d,k)},  \,  \frac{N}{C_d} \right]
$$ fully encompasses the interval $[N/(2C_d), N/C_d]$, and thus contains at least one prime. We therefore can assume that there is a prime $p$ satisfying 
\begin{equation}\label{eq:p range}
	\delta\ge 2\Gamma_d p^{-\nu(d, k)} \log p \mand  N \ge C_dp.
\end{equation}

Consider now a box $\fB(\bxi,\delta) \subseteq \Tor$. Clearly, the set of $\vu \in \F_p^d$ for which $\vu/p \in \fB(\bxi,\delta)$ forms a box $\fC_p(\bxi, \delta) \subseteq \F_p^d$ with side-length 
$$
	L\ge\lfloor p\delta\rfloor \ge \Gamma_d p^{1-\nu(d,k)}\log p.
$$ 
Let 
$$
	U_p(\bxi, \delta) = \# \left\{\vu \in \fC_p(\bxi, \delta) \cap \cZ_d:~ \left |T_{d, p} (\vu) \right|\ge \gamma_d p^{1/2}  \right\},
$$
where $\gamma_d$ is as in Lemma~\ref{lem:box dense}. From that lemma, we obtain in a straightforward manner the bound 
\begin{align*}
	U_p(\bxi, \delta) &\ge \# \left\{\vu \in \fC_p(\bxi, \delta):~ 
	\left |T_{d, p} (\vu) \right|\ge \gamma_d p^{1/2}  \right\} \\
	& \qquad \qquad \qquad \qquad - \#\{u_1 \in \F_p:~(u_1,0,\ldots,0) \in \fC_p(\bxi, \delta)\} \\
	&\gg L^d - L \gg (p\delta)^d.
\end{align*}
Therefore, if $\cN_p(\bxi, \delta)$ denotes the set of all $(x_1, \ldots x_d) \in \Tor$ having a diophantine approximation as in~\eqref{eq:approx log} with numerator $\vu$ counted by $U_p(\bxi, \delta)$, we have 
$$
	\lambda(\cN_p(\bxi, \delta)) \gg  \delta^d p^{d} \prod_{j=1}^d (N^j \log p)^{-1} = \delta^d p^{d}N^{-s(d)} \(\log p\)^{-d},
$$ 
and thus for any prime $p$ satisfying the conditions~\eqref{eq:p range} we have 
\begin{align*}
	I_{s, d}^{\flat}(\delta;N) & \gg \inf_{\bxi \in T^d}  \left(\lambda(\cN_p(\bxi, \delta)) \inf_{\vx \in \cN_p(\bxi, \delta)}|S_d(\vx;N)|^{2s} \right)\\
	&\gg\delta^d p^{d} N^{-s(d)}  \(N p^{-1/2} \)^{2s}  \(\log p\)^{-d} \\
	& \gg   \delta^d p^{d-s} N^{2s-s(d)} \(\log N\)^{-d}.
\end{align*} 

Recall now that by our assumption~\eqref{eq:delta large}, for a sufficiently large $N$ we can always find a prime $p$ satisfying~\eqref{eq:p range} with 
$$
	p \ll  \delta^{-1/\nu(d, k)} (\log N)^{1/\nu(d, k)}  
$$
as well as a prime $p$ (also satisfying~\eqref{eq:p range}) with
$$
	p \gg N. 
$$
Hence, under the condition~\eqref{eq:delta large} we have
\begin{align*}
	I_{s, d}^{\flat}(\delta;N) & \gg   \delta^d p^{d-s} N^{2s-s(d)} \(\log N\)^{-d} \\
	& \ge   \max \{  \delta^d  N^{s+d-s(d)} , \delta^{d -(d-s)/ \nu(d, k)} N^{2s-s(d)}\} N^{o(1)},
\end{align*}
which finishes the proof.

\section{Further comments}

\subsection{Mean values over more general sets}
Our setting involving multidimensional mean values opens up a certain degree of flexibility in terms of the shape of the underlying domain, and Wooley's conjecture (Conjecture~\ref{conj:Wooley}) admits for arbitrary measurable sets. Arguably, boxes of variable sidelength that reflects the distinct powers in the exponential sum might be better suited to understand the local behaviour of Weyl sums. Another approach is to investigate local behaviour only with respect to the coordinate corresponding to the highest degree, which contributes most of the oscillations of exponential sums. The case of boxes of the shape $[0,1)^{d-1} \times [0,\delta]$ has been studied in some detail in work by Demeter, Guth and Wang~\cite{DGW} as well as Guth and Maldague~\cite{GM} 
on small cap decouplings, extending previous work by Bourgain~\cite{Bour}. Even though in the work at hand we restricted our attention to hypercubes, our methods can be extended without serious problems to other axis-aligned boxes as well.

\subsection{Applications to the Schr\"odinger equation}
\label{sec:Schrodinger}

Our results have consequences for solutions of Schr\"odinger equations over short intervals.
The Schr\"odinger equation 
$$
2\pi u_t+iu_{xx}=0
$$ 
models the behaviour of quantum mechanical particles. We denote by $\rho(t,\cI)$ the probability that the particle belongs to the interval $\cI$ at time $t$. When $u(x,t)$ is a solution to the Schr\"odinger equation, then this probability is given by
\begin{equation}
\label{eq:prob}
\rho(t,\cI) = \int_{\cI} |u(x, t)|^2dx.
\end{equation}

In the case when the boundary condition is periodic of the shape 
$$
u(x, 0)=\sum_{n=1}^{N}a_n \e(xn),
$$
the solutions of the Schr\"odinger equation are trigonometric polynomials with quadratic amplitudes of the shape
$$
u(x, t)=\sum_{n=1}^{N} a_n\e(xn+tn^2).
$$

For a fixed $t\in \T$, our results do not yield any estimate for the  value~\eqref{eq:prob}. However, from our results we can deduce various upper and lower bounds on the above probability $\rho(t,\cI)$ for any short interval $\cI$ and for some time in yet another short interval.  

For example, in the case of the constant coefficients 
$
a_n =  1, n \in \N,
$
by Theorems~\ref{thm:mvt d=2} and~\ref{thm:lower d=2}, we have the following.

\begin{cor}\label{cor:distribution}
	Let $N\in \N$ be a large number, and let $(x_0, t_0)\in \T_2$. Then
	\begin{enumerate}
		\item for $\delta \ge N^{-3/8}$, there exists $t\in [t_0, t_0+ \delta]$ such that  
		$$
		\int_{x_0}^{x_0+\delta} \left |  \sum_{n=1}^{N} \e(xn+tn^2)\right |^2dx\le \delta N^{5/4+o(1)}; 
		$$
		\item  if $\delta \ge c/N $ for some small $c>0$, there exists $t\in [t_0, t_0+ \delta]$, such that   
		$$
		\int_{x_0}^{x_0+\delta} \left |  \sum_{n=1}^{N} \e(xn+tn^2)\right |^2dx\gg \delta.
		$$
	\end{enumerate}
\end{cor}

\begin{proof}
	Clearly, we have   
	\begin{align*}
	\delta  \min_{t \in [t_0, t_0+\delta]} \int_{x_0}^{x_0+\delta} \left |\sum_{n=1}^{N}  \e(xn+tn^2)\right |^2dx & \le \int_{t_0}^{t_0+\delta} \int_{x_0}^{x_0+\delta} 		\left |\sum_{n=1}^{N}\e(xn+tn^2)\right |^2dx \\
	& \le I_{2, 1}^{\sharp}\( \delta;N\).
	\end{align*}
	It thus suffices to observe that for the first statement, Theorem~\ref{thm:mvt d=2} with parameters $s=1$ and any $\delta \ge N^{-3/8}$ yields the bound 
	$$
	I_{2, 1}^{\sharp}(\delta;N) \le \delta^2 N^{2(1-3/8)+o(1)} =\delta^2 N^{5/4+o(1)},	
	$$
	which proves the claim (1). The second statement (2) is established similarly by combining the bound 
	$$
	\delta\max_{t \in [t_0, t_0+\delta]} \int_{x_0}^{x_0+\delta} \left |\sum_{n=1}^{N}  \e(xn+tn^2)\right |^2dx \ge I_{2, 1}^{\sharp}(\delta;N)
	$$
	with the bound 	
	$$
	I_{2, 1}^{\sharp}(\delta;N)\gg  \delta^2	
	$$	
	from Theorem~\ref{thm:lower d=2}\eqref{it:lower d=2 delta>>N}. 
\end{proof}

\section*{Acknowledgments} We are grateful to Roger Baker for his contributions in the initial phase of the paper. 

During the preparation of this manuscript, JB was supported by Starting Grant 2017-05110 and, in the final stages, Project Grant 2022-03717 of the Swedish Science Foundation 
$\mathrm{(Vetenskapsr\mathring{a}det)}$, CC was supported by the National Natural Science Foundation of China Grant 12101002, and IS was supported by the Australian Research Council Grant DP170100786.  Part of the work was completed while JB and IS were in residence at the Max-Planck-Institute for Mathematics in Bonn, whose generous support and excellent working conditions are also gratefully acknowledged.


\begin{thebibliography}{www}
 
 

\bibitem{Bak0} R. C. Baker, `Weyl sums and Diophantine approximation', \textit{J. London Math. Soc.} \textbf{25}   (1982), 25--34;
 Correction, \textit{ibid.} \textbf{46} (1992), 202--204.

\bibitem{Bak1} R. C. Baker, `Small fractional parts of polynomials', \textit{Funct. et Approx.} \textbf{55} (2016), 131--137.

\bibitem{BCS1} R. C. Baker, C. Chen and I. E. Shparlinski,  
`Bounds on the norms of maximal operators on Weyl sums', 
\textit{Preprint}, 2021, available at \url{https://arxiv.org/abs/2107.13674}.


\bibitem{BCS2} R. C. Baker, C. Chen and I. E. Shparlinski,  
`Large Weyl sums and Hausdorff dimension', 
 \textit{J. Math. Anal. Appl.} \textbf{510} (2022), Art.~26030.

\bibitem{Bour} J. Bourgain, `Decoupling inequalities and some mean-value theorems', 
 \textit{J.  d’Anal.  Math.},  \textbf{133} (2017), 313--334.

\bibitem{BDG} J. Bourgain, C. Demeter and L. Guth, 
`Proof of the main conjecture in Vinogradov's mean value theorem for degrees higher than three', 
\textit{Ann.\ Math.}  \textbf{184} (2016), 633--682. 

\bibitem{BrPa} J. Brandes and S. T. Parsell, `The Hasse principle for diagonal forms restricted to lower-degree hypersurfaces', \textit{Algebra and Number Theory}, \textbf{15} (2021), No. 9, 2289--2314

\bibitem{BrHu} J. Brandes and K. Hughes, `On the inhomogeneous Vinogradov system', 
 \textit{Bull. Aust. Math. Soc.},  \textbf{106} (2022), 396--403. 


\bibitem{ChSh-IMRN} C. Chen and I. E. Shparlinski,
`New bounds of Weyl sums',  
 \textit{Intern. Math. Res. Notices} \textbf{2021}  (2021),  8451--8491. 
 
\bibitem{CKMS} C. Chen, B. Kerr, J. Maynard and I. E. Shparlinski, 
 `Metric theory of Weyl sums', \textit{Math. Ann.} \textbf{385} (2023), 
 309--355.

\bibitem{ChSh-AM} C. Chen and I. E. Shparlinski,
`On large values of Weyl sums',  
\textit{Adv. Math.}  \textbf{370} (2020), Art.~107216.

\bibitem{ChSh-QJM} C. Chen and I. E. Shparlinski,
`Restricted mean value theorems and 
metric theory of restricted Weyl sums', 
 \textit{Quart. J. Math.\/},    \textbf{72} (2021), 885--919.


\bibitem{ChSh-JMAA} C. Chen and I. E. Shparlinski,
`Small values of Weyl sums', 
 \textit{J. Math. Anal. Appl.}    \textbf{495}  (2021),  Art.~124743. 
 
\bibitem{CPR}  T. Cochrane, C. Pinner and J. Rosenhouse, 
`Sparse polynomial exponential sums',
 \textit{Acta Arith.} \textbf{108} (2003),  37--52.
 
 \bibitem{DGW} C. Demeter, L. Guth and H. Wang, `Small cap decouplings', \textit{Geom. Funct. Anal.} \textbf{20} (2020), 989--1062.
 
  \bibitem{DeLa} 
C.  Demeter and B. Langowski, 
`Restriction of exponential sums to hypersurfaces', 
\textit{Intern. Math. Res. Notices} (to appear).  

\bibitem{DHV} S. Dendrinos, K. Hughes and M. Vitturi, `Some subcritical estimates for the $\ell^p$-improving problem for discrete curves', \textit{J.  Fourier Anal. Appl.},  \textbf{28} (2022), Art.~69.

 \bibitem{FJK} H. Fiedler, W. Jurkat and O. K\"orner, `Asymptotic expansions of finite theta series',  \textit{Acta Arith.} \textbf{32} (1977), 129--146.
 

\bibitem{GM} L. Guth and D. Maldague, `Small cap decoupling for the moment curve in $\mathbb{R}^3$', 
\textit{Preprint}, 2022, available at \url{https://arxiv.org/abs/2206.01574}.
 
 
\bibitem{IwKow} H. Iwaniec and E. Kowalski,
\textit{Analytic number theory}, Amer.  Math.  Soc.,
Providence, RI, 2004. 


\bibitem{KnSok1} L. A. Knizhnerman and V. Z. Sokolinskii, `Some estimates for rational trigonometric sums and sums of Legendre symbols',  \textit{Uspekhi Mat. Nauk} \textbf{34} (3)  (1979), 199--200 (in Russian); 
translated in  \textit{Russian Math. Surveys} \textbf{34} (3) (1979), 203--204. 

\bibitem{KnSok2} L. A. Knizhnerman and V. Z. Sokolinskii, `Trigonometric sums and sums of Legendre symbols with large and small absolute values',   \textit{Investigations in Number Theory}, Saratov, Gos. Univ., Saratov, 1987,  76--89 (in Russian). 


 \bibitem{Li}  W.-C. W. Li,  \textit{Number theory with applications},
World Scientific,  Singapore, 1996.

 
\bibitem{LN} R. Lidl and H. Niederreiter,  \textit{Finite Fields}, 
Cambridge Univ. Press, Cambridge, 1997.


\bibitem{Mont} H. L. Montgomery, \textit{Ten lectures on the interface between
 analytic number theory and harmonic analysis},  Amer. Math. Soc.,  Providence, RI, 1994.

\bibitem{Vau} R. C. Vaughan, \textit{The Hardy-Littlewood method}, Cambridge University Press, 1997.
 

\bibitem{Wool1} T.~D.~Wooley,
`The cubic case of the main conjecture in Vinogradov's mean value theorem',  
\textit{Adv.  in Math.} \textbf{ 294} (2016), 532--561

\bibitem{Wool2} T.~D.~Wooley, 
`Nested efficient congruencing and relatives of Vinogradov's mean value theorem',
\textit{Proc. London Math. Soc.} \textbf{118} (2019), 942--1016.

\bibitem{Wool3} T.~D.~Wooley, 
`Subconvexity in inhomogeneous Vinogradov systems', 
 \textit{Quart. J. Math.}, (to appear). 



\end{thebibliography}
  \end{document}